\documentclass[11 pt]{amsart}

\usepackage{amsthm,amssymb,amstext,amscd,amsfonts,amsbsy,amsxtra,latexsym,cite}
\usepackage[latin1]{inputenc}
\usepackage{graphicx, color}

\newcommand{\field}[1]{\mathbb{#1}}

\newcommand{\Z}{\field{Z}}
\newcommand{\R}{\field{R}}

\newcommand{\Prob}{\mathbf{P}}
\newcommand{\calC}{\mathcal{C}}

\newcommand{\calR}{\mathcal{R}}
\newcommand{\calS}{\mathcal{S}}
\newcommand{\Dk}{\mathcal{D}_k}
\newcommand{\Jk}{\mathcal{J}_k}
\newcommand{\tilHk}{\widetilde{H}_k}
\newcommand{\Shell}{\text{Shell}}

\theoremstyle{plain}
\newtheorem{theorem}{Theorem}
\newtheorem*{theorem*}{Theorem}
\newtheorem{lemma}[theorem]{\textbf{Lemma}}
\newtheorem{proposition}[theorem]{\textbf{Proposition}}
\newtheorem{corollary}[theorem]{\textbf{Corollary}}
\newtheorem*{conjecture*}{Conjecture}

\theoremstyle{definition}
\newtheorem{definition}[theorem]{Definition}

\theoremstyle{remark}
\newtheorem*{remark}{Remark}

\numberwithin{theorem}{section} \numberwithin{equation}{section}

\renewenvironment{proof}[1][Proof]{\begin{trivlist}
\item[\hskip \labelsep {\bfseries #1:}]}{\qed\end{trivlist}}

\begin{document}

%\begin{center}\underline{\Large Preliminary Version} \\
%\end{center}

\bigskip

\bigskip

\bigskip

\title[Improved bounds for generalized bootstrap percolation]{Improved bounds on metastability thresholds and probabilities for generalized bootstrap percolation}

\author{Kathrin Bringmann}
\address{Mathematical Institute\\University of
Cologne\\ Weyertal 86-90 \\ 50931 Cologne \\Germany}
\email{kbringma@math.uni-koeln.de}
\author{Karl Mahlburg}
\address{Department of Mathematics \\
Princeton University \\
NJ 08544\\ U.S.A.} \email{mahlburg@math.princeton.edu}
\date{\today}

\thanks{The authors thank the Mathematisches Forschungsinstitut Oberwolfach for hosting this research through the Research in Pairs Program.  The first author was partially supported by NSF grant DMS-0757907 and by the Alfried Krupp prize.
The second author was supported by an NSF Postdoctoral Fellowship
administered by the Mathematical Sciences Research Institute through
its core grant DMS-0441170. }

\begin{abstract}
We generalize and improve results of Andrews, Gravner, Holroyd,
Liggett, and Romik on metastability thresholds for generalized
two-dimensional bootstrap percolation models, and answer several of their open
problems and conjectures.  Specifically, we prove slow convergence
and localization bounds for Holroyd, Liggett, and Romik's
$k$-percolation models, and in the process provide a unified and
improved treatment of existing results for bootstrap, modified
bootstrap, and Frob\"ose percolation. Furthermore, we prove improved
asymptotic bounds for the generating functions of partitions without
$k$-gaps, which are also related to certain infinite probability
processes relevant to these percolation models.

One of our key technical probability results is also of independent
interest.  We prove new upper and lower bounds for the probability
that a sequence of independent events with monotonically increasing
probabilities contains no ``$k$-gap'' patterns, which interpolates
the general Markov chain solution that arises in the case that all
of the probabilities are equal.
\end{abstract}

\maketitle

\section{Introduction and statement of results}
In \cite{HLR}, Holroyd, Liggett, and Romik considered two-dimensional {\it
$k$-percolation} models for $k \geq 2$.  In this cellular automaton
model, an initial configuration on $\Z^2$ is randomly determined by
independently setting each site to be active with probability $1-q$,
or empty (inactive) with probability $q$.  Throughout we will write
$q := e^{-s}$.  Active sites always remain active, and the system
evolves by following a {\it threshold} growth rule: if $N_k(x)$ contains
at least $k$ active sites, then $x$ becomes active, where the
neighborhood is the $(k-1)$-cross given by
\begin{equation}
\label{E:Nk} N(x) = N_k(x) := \left\{x + w \; : \; w = (v,0) \text{
or } (0,v), -(k-1) \leq v \leq (k-1), \: v \neq 0\right\}.
\end{equation}
The case $k=2$ is the widely studied {\it bootstrap percolation}
model \cite{AL, Hol03}.  Other commonly studied models include {\it
modified} bootstrap percolation, as well as the {\it Frob\"ose}
model \cite{ASA, AL, Hol03}, which are frequently treated in
conjunction with bootstrap percolation.  Indeed, we will see
throughout this paper that these latter two models can be
fundamentally thought of as being associated with the ``degenerate''
$k=1$ case.

One of the important questions in the study of such models is
whether a given initial configuration will eventually fill the
entire plane with active sites (if not, then there exist sites that
will never become active).  In fact, in light of Schonmann and van Enter's proofs
that the critical probability for percolation on the plane is zero \cite{Sch, vanE}, this question is more properly
asked on finite square regions instead.  A seminal paper by
Holroyd \cite{Hol03} showed that there is a precise {\it metastability
threshold} for bootstrap percolation, which means that the most
interesting (critical) behavior occurs when the probability $q$ and the
side-length $L$ satisfy an exponential relationship.  The limiting
threshold comes into play as $s \to 0$, $q \to 1$, and,
simultaneously, $L \to \infty$  (the rate of exponential scaling is
the {\it critical exponent}).  It should also be noted that Holroyd's
results built upon earlier work of Aizenman and Lebowitz \cite{AL}
that proved the existence of (possibly unequal) upper and lower
threshold bounds.

Subsequently, Holroyd, Liggett, and Romik also exactly described the critical
exponents for each of the $k$-percolation models.  These models form an infinite
family of two-dimensional models that vary in the {\it neighborhood} aspect, but one can also consider percolation
models with similar neighborhoods that vary in {\it dimension}.  This is the subject of recent and ongoing work of Holroyd \cite{Hol06}, and also Balogh, Boll\'obas, and Morris \cite{BBM10, BBM09}, who have found precise critical exponents
for nearest neighbor bootstrap percolation models in all dimensions.

Such metastability thresholds can be further understood through the study of ``localized'' models,
where all active sites must emanate from a fixed initial location
(typically the origin).  In particular, once the global critical exponent is known,
a better understanding of localized growth can lead to refined estimates for the
rate of convergence in the limit.  Gravner and Holroyd used this approach to prove slow convergence
estimates for bootstrap percolation in \cite{GH09}.

In this paper we  generalize Gravner and Holroyd's work to all $k$-percolation models.  We define a local version of
$k$-percolation that has three possible states for each cell: {\it
active}, {\it occupied}, or {\it empty.} An initial configuration
$\calC$ is generated by letting the origin be active  with
probability $1-q$ and empty with probability $q$; all other sites
are either occupied with probability $1-q$ or empty with probability
$q$. Throughout we denote the corresponding probability measure by $\Prob$.

If $k > 1$, then the growth rules are the following:
\begin{itemize}
\item
An occupied site becomes active if there is at least one active site
within $\ell^1$-distance
 $k$.
\item
An empty site $x$ becomes active if there are at least $k$ active
sites in $N_k(x)$.
\end{itemize}
For $k=1$, there are two different models. For the modified local
model we have the following rules:
\begin{itemize}
\item
Each occupied site becomes active if there is at least
one active site within $\ell^\infty$-distance $1$.
\item
An empty site $x$ becomes active if there is at least
one active site in each of $\{x \pm (0,1)\}$ and $\{x \pm (1,0)\}$.
\end{itemize}
For the Frob\"ose local model we have the rules:
\begin{itemize}
\item
Each occupied site becomes active if there is at least
one active site within $\ell^1$-distance $1$.
\item
 An empty site $x$ becomes active if there are two
active sites as described in the second rule of the  modified case,
and if the cell in the ``corner'' between the two sites is also
active.
\end{itemize}

The results on metastability thresholds are closely related to the
concept of {\it indefinite growth} in localized percolation models,
which means that every site in $\Z^2$ eventually becomes active.
For the remainder of this paper we focus only on this perspective.
Along these lines, Gravner and Holroyd proved the following bounds
for local bootstrap percolation, modified bootstrap percolation, and
Frob\"ose percolation.
\begin{theorem*}[Theorem 1 in \cite{GH09}]
If $k = 1$ (modified model) or $2$, then there exist positive
constants $c_1, c_2,s_0$ such that for $s<s_0$, we have
\begin{equation*}
\exp\left(-2\lambda_ks^{-1}+c_1 s^{-\frac{1}{2}}\right)\leq
\Prob(\text{indefinite growth})\leq \exp\left(-2\lambda_ks^{-1} +c_2
s^{-\frac{1}{2}}\left(\log s^{-1}\right)^{3}\right).
\end{equation*}
If $k = 1$ (Frob\"ose model), then the upper bound has the power
$\left(\log s^{-1}\right)^2$ instead.  Here the threshold constants
are given as
\begin{equation*}
\lambda_k:=\frac{\pi^2}{3k(k+1)}.
\end{equation*}
\end{theorem*}
\begin{remark}
Our notation loosely follows that of \cite{And05}, and many of the
results in the prior literature are stated in terms of the
probability parameter $q$ rather than the exponential parameter $s$.
However, all such statements are equivalent, since  $q = s + O(s^2)$
as $s \to 0$.
\end{remark}

Gravner and Holroyd conclude their papers on local percolation by posing several
questions, which include the following problems:
\begin{itemize}
\item
Extend the results to other bootstrap percolation models for which
sharp thresholds are known to exist, including bootstrap
$k$-percolation (question (iii) of \cite{GH09}).
\item
Is a power of $\log s^{-1}$ in the upper bound really necessary
(question (i) of \cite{GH08})?
\end{itemize}

We solve the first of these problems by proving a new, unified
result for all $k \geq 1$, which includes both the modified and
Frob\"ose models in the $k = 1$ case.  We also make progress on the
second question by showing that the power of the logarithm in the
upper bound is at most $5/2$, which improves the power of $3$ found
in Gravner and Holroyd's results for $k=1$ and $2$ (we reduce the
power from $2$ to $1$ in the Frob\"ose model).
\begin{theorem}
\label{T:maintheorem} For each $k \geq 1$, there exists a
sufficiently small $n_0$, and constants $c_1, c_2$ such that for $s>
n_0$
 we have
\[
\exp\left(-2\lambda_ks^{-1}+c_1 s^{-\frac{1}{2}}\right)\leq
\Prob(\text{indefinite growth})\leq \exp\left(-2\lambda_ks^{-1} +c_2
s^{-\frac{1}{2}}\left(\log s^{-1}\right)^{\frac{5}{2}}\right).
\]
For the $k=1$ Frob\"ose model, the upper bound has the power $\log
s^{-1}$ instead.
\end{theorem}
\begin{remark}
Throughout the later sections of this paper we use
lower-case $c$'s to represent various (effective) constants without distinguishing them by different indices.
Although this is a slight abuse of notation, there are a finite
number of such constants in use, and one can simply adopt the
supremum/infimum of the collections as necessary.
\end{remark}

There are two main components underlying our improvements.  First,
the combinatorial arguments used in Gravner and Holroyd's
upper bound results for $k=1$ and $2$ do not immediately generalize
to higher $k$, and we adopt a significantly different conceptual
approach in order to unify all cases.  Much of the literature in this subject, including \cite{AL, Hol03,
GH09, GH08, HLR}, uses the concept of {\it internally spanned}
rectangles when studying upper bounds; the general idea being that
if a large rectangle  eventually becomes active, then there
must be a nested subsequence of smaller rectangles that each became
active during the overall growth process.

Our new approach comes from the observation that in their final
form, these previous arguments replace internally spanned rectangles by rectangles
that satisfy certain row and column conditions.  In other
words, the combinatorial approach is unnecessarily ``tighter''
than the bounds that are actually proven, and the arguments are
therefore more restrictive (and thus more complicated) than is
necessary.  We introduce a new combinatorial construction that is more precisely
tailored to reflect the row and column conditions that are actually used in approximations.
The combinatorial approach used in the proof of the $k=1$ and $k=2$ cases of the
lower bound of Theorem \ref{T:maintheorem} then generalize in a more straightforward manner.

The second new component of our arguments is a two-sided bound for the
probabilities of certain pattern-avoiding sequences that naturally
underlie the preceding applications to bootstrap percolation models
(a very rough approximation of these bounds was proven and used in
\cite{HLR}). For $s \in (0,1)$, let $C_1, \ldots, C_n, \ldots$ be
independent events with probabilities
\begin{equation*}
\Prob(C_n)=1-e^{-ns}.
\end{equation*}
Let $A_k$ be the event that there are no $k$-gaps among the
occurrences of the events $C_i$, which means that there are no consecutive $C_i$'s that do not occur.
Symbolically, this can be written
as
\begin{equation}
A_k:=\bigcap_{i=1}^\infty \left(C_i\cup C_{i+1}\cup \ldots \cup
C_{i+k-1}\right).
\end{equation}
Holroyd, Liggett, and Romik described the logarithmic limiting
behavior of these probabilities.
\begin{theorem*}[Theorem 2 in \cite{HLR}]
For every $k \geq 1$,
\begin{equation*}
-\log \Prob(A_k) \sim \lambda_k s^{-1} \qquad \qquad \text{as } s
\to 0.
\end{equation*}
\end{theorem*}
\begin{remark}
Although this is not stated in the paper, equation (14) from Section
3 of \cite{HLR} immediately implies that if $s$ is sufficiently
small, then there exists a positive constant $c_2$ such that
\begin{equation}
\label{E:HLRupper} -\log \Prob(A_k) \leq \lambda_k s^{-1} + c_2
s^{-\frac{1}{2}} \log{s^{-1}}.
\end{equation}
Equation (14) in \cite{HLR} also leads to a corresponding implicit
lower bound, although an additional technical result is needed.
Using Lemma \ref{L:gkproperties} part (ii) from later in the
present paper, the resulting bound would state that there is a positive
constant $c_1$ such that
\begin{equation}
\label{E:HLRlower} \lambda_k s^{-1} - c_1 s^{-\frac{1}{2}}
\log{s^{-1}} \leq -\log \Prob(A_k).
\end{equation}
\end{remark}

The probability events $A_k$ are also of interest in combinatorial
number theory and the theory of partitions, and Holroyd, Ligget, and
Romik explained the connections in \cite{HLR}.  They showed that
\begin{equation}
G_k(q) = \sum_{n \geq 0} p_k(n) q^n = \frac{\Prob(A_k)}{ (q;
q)_{\infty}} =
 \Prob(A_k) \cdot \sum_{n \geq 1} p(n) q^n,
\end{equation}
where $p(n)$ denotes the number of integer partitions of $n$, and
$p_k(n)$ denotes the number of partitions of $n$ without
$k$-sequences, i.e., those partitions that do not contain any $k$
consecutive integers as parts.  We have also used the standard notation
$(a;q)_n=(a)_n := \prod_{j=1}^{n-1} \left(1-aq^j \right)$ for the
rising $q$-factorial.

These partitions were further studied by
Andrews in \cite{And05}, who found the explicit (double
hypergeometric) $q$-series expansion
\begin{equation}
G_k(q) = \frac{1}{ (q; q)_{\infty}} \sum\limits_{r,s \geq 0}
\frac{(-1)^s q^{\binom{k+1}{2}(s+r)^2 + (k+1)\binom{r+1}{2}}}{
\left( q^k;q^k\right)_s \left( q^{k+1};q^{k+1}\right)_r }.
\end{equation}
In the case $k=2$, he also found the alternative expression
\begin{equation}
G_2(q) =  \frac{\left( -q^3;q^3\right)_{\infty}}{\left(q^2;q^2
\right)_{\infty}} \cdot \chi(q),
\end{equation}
where
\begin{equation*}
\chi(q) := 1+ \sum\limits_{n \geq 1} \frac{q^{n^2}}{\prod\limits_{j
= 1}^n \left(1-q^j + q^{2j}\right)}
\end{equation*}
is one of Ramanujan's third-order mock theta functions \cite{Wat}.
He then used the cuspidal expansions of mock theta functions and
modular forms to prove the exact (non-logarithmic) asymptotic
formula in the case $k = 2$:
\begin{equation*}
\Prob(A_2) \sim \sqrt{\frac{\pi}{2}} s^{-\frac{1}{2}}
\exp{\left(-\lambda_2 s^{-1}\right)} \qquad \qquad \text{as } s \to
0.
\end{equation*}
\begin{remark}
The present authors showed in \cite{BM} that although this series
asymptotic corresponds to the ``dominant cusp'' for $G_2(q)$, it
does not provide sufficient information to determine the full
asymptotic expansion of the $q$-series coefficients $p_2(n)$, or
even the entire leading exponential term.  The authors developed and
used an amplified version of the Hardy-Ramanujan circle method in
order to find formulas for the coefficients of such functions
(general products of mock theta functions and modular forms), with
error at most $O(\log n)$.
\end{remark}

Although the functions $G_k(q)$ do not seem to be automorphic forms
for $k \geq 3$, Andrews also found numerical evidence for the
asymptotic in the general case and identified theta function
components that make asymptotically ``large'' contributions.  This
led him to a conjectural formula for the general asymptotic.
\begin{conjecture*}[Andrews \cite{And05}]
For each $k \geq 2$, there is a positive constant $c_k$ such that
\begin{equation*}
\Prob(A_k) \sim c_k s^{-\frac{1}{2}} \exp\left(-\lambda_k
s^{-1}\right) \qquad \qquad \text{as } s \to 0.
\end{equation*}
\end{conjecture*}

We greatly refine Holroyd, Liggett, and Romik's theorem for the logarithmic
asymptotic of $\Prob(A_k)$ (as well as the implicit bounds \eqref{E:HLRupper} and \eqref{E:HLRupper}), and also make
significant progress toward Andrews' conjecture by proving
(non-logarithmic) lower and upper bounds that differ by less than a
multiplicative factor of $s$.
\begin{theorem}\label{T:probTheorem}
For every positive integer $k$, we have the following asymptotic as
$s \to 0$:
\begin{equation*}
\exp\left(-\lambda_ks^{-1}\right)\leq \Prob(A_k)\leq
s^{-\frac{(2k-1)}{2k}}\exp\left(-\lambda_ks^{-1}\right)\left(1+o(1)\right).
\end{equation*}
\end{theorem}
\begin{remark}
In fact, our proof shows that $\Prob(A_k)$ lies between lower and
upper bounds that differ by a multiplicative factor of at most
$s^{-(k-1)/k}$; however, in order to write both bounds in a clean
manner, we have expanded the range to $s^{-(2k-1)/2k}$ in the theorem
statement.
\end{remark}
One very interesting feature of this result is that we have used
arguments from combinatorics and probability in order to conclude
better asymptotic bounds for the coefficients of a $q$-series.  In the subject of
combinatorial $q$-series, such implications more frequently proceed
in the opposite direction. For example, Andrews'   work on the $k=2$
case used analytic and automorphic properties of modular forms and
mock theta functions in order to conclude more precise combinatorial
and probabilistic statements \cite{And05, BM}.

The remainder of the paper is organized as follows.  In Section
\ref{S:Probability} we define a class of certain pattern-avoiding
probability sequences and prove an important new general bound for
the probability of sequences without $k$-gaps.  In Section
\ref{S:Logprobability} we turn to a logarithmic version of the
function from Section \ref{S:Probability} and prove several useful
estimates that are used in bounds throughout the rest of the paper.
The proof of Theorem \ref{T:maintheorem} is found in the next two
sections, and is split into the lower and upper bounds.  Section
\ref{S:Lower} contains the proof of the lower bound in Theorem
\ref{T:maintheorem}, which combines generalized versions of the
combinatorial ideas from \cite{GH08} with the estimates from Section
\ref{S:Logprobability}.  The upper bound then follows in Section
\ref{S:Localupper}; in order to adapt Gravner and Holroyd's
arguments in \cite{GH09} to the case of general $k$, we introduce a
conceptually different combinatorial model for tracking the possible
ways a configuration can grow.  We also make slight (and optimal)
modifications to some of the parameters used in Gravner and
Holroyd's arguments in order to achieve an improved
$\log s^{-1}$-power.  Finally, Theorem \ref{T:probTheorem} is proved in Section
\ref{S:PAk} by adapting the general probability results of Section
\ref{S:Probability} to the specific cases described in the theorem
statement.

\section*{Acknowledgments}
To be entered later.

%The authors thank Michael Aizenman, Aernout van Enter, Rob Morris, and Dan Romik
%for helpful discussions regarding the history of bootstrap percolation, as well as for informing us
%of other recent developments in the subject.

\section{Probability results}
\label{S:Probability}

\subsection{Definitions and notation}

In \cite{HLR}, Holroyd, Liggett, and Romik introduced and studied an interesting
family of functions.  For each positive integer $k$,
$f_k(x)$ is defined to be the unique, decreasing function on $[0,1]$
that satisfies the functional equation
\begin{equation}\label{E:fshortform}
f^k(x)-f^{k+1}(x)=x^k-x^{k+1}.
\end{equation}
\begin{remark}
Note that $f_k(0)=1$ and $f_k(1)=0$ for all $k$, and also that the
first two cases have simple explicit equations; namely, $f_1(x) =
1-x$ and $f_2(x) = \frac{1-x + \sqrt{(1-x)(1+3x)}}{2}$ (this case
was studied extensively by Holroyd in \cite{Hol03}).
\end{remark}

As was pointed out in \cite{HLR}, the existence and (Lipschitz)
continuity of $f_k$ follows immediately from the fact that $h_k(x)
:= x^k - x^{k+1}$ is strictly increasing for $0 \leq x <
\frac{k}{k+1}$ and is strictly decreasing for $\frac{k}{k+1} < x
\leq 1.$ We will also require the following equivalent ``long'' form
of \eqref{E:fshortform}:
\begin{equation}\label{E:flongform}
f^k = (1-x)\left(f^{k-1} + xf^{k-2} + \dots + x^{k-1}\right).
\end{equation}

The main result in this section is a fairly tight probability bound
that is used in many places throughout the paper.  Note that a
sequence of probabilistic events is said to have a {\it $k$-gap} if
there are $k$ consecutive events that do {\it not} occur.
\begin{theorem}
\label{T:kgaps2sided} Suppose that $A_1, \ldots, A_n$  are
independent events that each occur with probability $\Prob(A_i) =
u_i$, where $u_1, \ldots, u_n$ form a monotonically increasing
sequence.  Then this sequence of events satisfies
\begin{equation*}
\prod\limits_{i=1}^n f_k(1-u_i) \leq \Prob\Big(\{A_i\}_{i=1}^n\text{
has no $k$-gaps}\Big) \leq \prod_{i=k}^n f_k(1-u_i).
\end{equation*}
\end{theorem}
\begin{remark}
If $n < k$, then the empty product on the right-hand side should be
interpreted as having the value $1$.
\end{remark}
\begin{remark}
If we take the limit $n \rightarrow \infty$, then the case of monotonically increasing probabilities $u_i$ is
essentially the only interesting possibility; if the probabilities oscillate indefinitely, then all of the terms in Theorem \ref{T:kgaps2sided} approach zero rapidly.
\end{remark}
The proof of both bounds will use inductive arguments.  Define the
shorthand notation
\begin{equation*}
\rho_n:= \Prob\Big(\{A_i\}_{i=1}^n\text{ has no $k$-gaps}\Big),
\end{equation*}
and observe that the initial $k$ values are clearly
$\rho_0=\ldots=\rho_{k-1}=1$, all of which satisfy the theorem
bounds.

The inductive argument for all other cases follows from a simple
combinatorial recurrence.  Observe that if the event $\rho_{n+k}$
occurs (where $n \geq 0$), then at least one of $A_{n+1}, \dots,
A_{n+k}$ must occur.  Separating into distinct cases based on the
final such $A_{n+i}$ leads to the following recurrence:
\begin{equation}\label{E:rec}
\rho_{n+k}=\rho_{n+k-1}u_{n+k}+\rho_{n+k-2}u_{n+k-1}(1-u_{n+k})+\dots+\rho_n
u_{n+1}(1-u_{n+2})\dots(1-u_{n+k}).
\end{equation}
The first term corresponds to the case where $A_{n+k}$ occurs, the
second term the case where $A_{n+k-1}$ occurs and $A_{n+k}$ does
not, and so on until the final term, which represents the case where
only $A_{n+1}$ occurs.
\begin{remark}
If $k=1$, Theorem \ref{T:kgaps2sided} is actually an equality (which
is seen to be trivially true upon recalling that $f_1(x) = 1-x$). Gravner
and Holroyd proved the lower bound for the case $k=2$ in \cite{GH08}
by using the explicit formula for $f_2(x)$. Furthermore, Holroyd,
Liggett, and Romik \cite{HLR} also proved the theorem for general
$k$ in the very special case that all $u_i$ are equal; their
argument essentially treats \eqref{E:rec} as a linear recurrence and
shows that $f_k$ gives the largest eigenvalue. Alternatively, this
is equivalent to calculating the limiting entropy of the
corresponding Markov process as $n$ increases.
\end{remark}

\subsection{Lower bound for probabilistic $k$-gaps}
\label{S:Lowerkgaps}

We prove a result that is more general than the lower bound in
Theorem \ref{T:kgaps2sided}, as it also allows for decreasing
probabilities.
\begin{proposition}\label{P:kgapslower}
Suppose that $A_1, \ldots, A_n$  are independent events whose
corresponding probabilities $u_1, \ldots, u_n$ are either increasing
or decreasing.  Then this sequence of events satisfies
\[
\Prob\Big(\{A_i\}_{i=1}^n\text{ has no $k$-gaps}\Big)\geq
\prod_{i=1}^n f_k(1-u_i).
\]
\end{proposition}

This proposition will follow from an inductive argument that relies
on an auxiliary function that is essentially an
approximation of \eqref{E:rec}.  For $k\geq 1$, let
\begin{align*}
H_k(y_1, \ldots, y_k)&:=
(1-y_k)f_k(y_1)\cdots f_k(y_{k-1})+(1-y_{k-1})y_kf_k(y_1)\cdots f_k(y_{k-2})\\
& +\cdots+(1-y_i)y_{i+1}\cdots y_k f_k(y_1)\cdots
f(y_{i-1})+\cdots+(1-y_1)y_2\cdots y_k-f_k(y_1)\cdots f_k(y_k).
\end{align*}
Note that $H_k(y, \ldots, y)=0$ by \eqref{E:flongform}. The most
important property of this function is found in the following
proposition.
\begin{proposition}\label{L:HLemma}
For increasing arguments $0\leq y_1\leq \ldots \leq y_k\leq 1$, the
function $H_k(y_1, \ldots, y_k)$ is nonnegative.
\end{proposition}

We further postpone the proof of Proposition \ref{L:HLemma} until
after we have proven several necessary intermediate results.
\begin{lemma}\label{L:fkLemma} \mbox{}
Consider the domain $0 \leq y \leq 1$.
\begin{enumerate}
\item
\label{L:fkLemma:fk}
The function $\displaystyle \frac{y}{f_k(y)}$ is increasing.
\item
\label{L:fkLemma:fk'}
We have
\begin{equation*}
\frac{f'_k(y)}{f_k(y)}\geq -\frac{1}{1-y}.
\end{equation*}
\end{enumerate}
\end{lemma}

\begin{proof}
\begin{enumerate}
\item
This holds because the functions $y$ and $\frac{1}{f_k(y)}$ are both increasing and non-negative.\\
\item Differentiating \eqref{E:flongform}, we find that
\begin{equation}\label{fkquotient}
\begin{split}
f'_k &\left(k f_k^{k-1}-(1-y)(k-1)f_k^{k-2}-\ldots-(1-y)y^{k-2}\right)\\
& = -f_k^{k-1}+(1-2y)f_k^{k-2}+\ldots+\left((k-1)y^{k-2}-ky^{k-1}\right)\\
\mathop{\iff}^{\text{\eqref{E:flongform}}}
\frac{f'_k}{f_k}& \left((1-y)f_k^{k-1}+2(1-y)yf_k^{k-2}+\ldots+(k-1)(1-y)y^{k-2}f_k+k(1-y)y^{k-1}\right)\\
& = -f_k^{k-1}+(1-2y)f_k^{k-2}+\ldots+\left((k-1)y^{k-2}-ky^{k-1}\right)\\
\iff
\frac{f'_k}{f_k}&=-\frac{1}{1-y}+\frac{f_k^{k-2}+2yf_k^{k-3}+\ldots+(k-1)y^{k-2}}{(1-y)\left(f_k^{k-1}+2yf_k^{k-2}+\ldots+ky^{k-1}\right)}.
\end{split}
\end{equation}
\end{enumerate}
This easily gives the claim.
\end{proof}

\noindent Next, for $1\leq j\leq k$ define
\begin{equation}
\label{E:Tj} T_{k, j}(y)=T_j:=\frac{(1-y)y^{j-1}}{f_k^j(y)}.
\end{equation}
This family of functions interpolates between $T_1$ and $T_k$, whose
opposing behaviors are described in the next result.
\begin{lemma}\label{TLemma}
The following hold for every $k \geq 1$.
\begin{enumerate}
\item
\label{TLemma:T1}
The function $T_1$ is decreasing.
\item
\label{TLemma:Tk}
The function $T_k$ is increasing.
\end{enumerate}
\end{lemma}

\begin{proof}
\begin{enumerate}
\item
We differentiate (\ref{E:Tj}) to get
\[
T'_1=\frac{1}{f_k}\left(-1+(1-y)\left(\frac{-f'_k}{f_k}\right)\right),
\]
and use Lemma \ref{L:fkLemma} \ref{L:fkLemma:fk'} to conclude that this is
non-positive.
\item
We have
\[
T'_k=\frac{(k-1)y^{k-2}-ky^{k-1}}{f_k^k}-\frac{kf'_k(y^{k-1}-y^k)}{f_k^{k+1}}.
\]
Thus $T'_k\geq 0$ is equivalent to
\begin{equation}\label{Tineq}
-\frac{k f'_k(y-y^2)}{f_k}\geq -(k-1)+ky\iff \frac{-f'_k}{f_k}\geq
\frac{ky-(k-1)}{ky(1-y)}.
\end{equation}
Inserting \eqref{fkquotient} implies that \eqref{Tineq} is
again equivalent to
\begin{align*}
& \frac{1}{1-y}-\frac{f_k^{k-2}+2y
f_k^{k-3}+\ldots+(k-1)y^{k-2}}{(1-y)\left(f_k^{k-1}+2y
f_k^{k-2}+\ldots+ky^{k-1}\right)} \geq \frac{ky-(k-1)}{ky(1-y)},
\end{align*}
which further simplifies to
\begin{align*}
& (k-1)f_k^{k-1}+2(k-1)yf_k^{k-2}+\ldots+(k-1)k y^{k-1} \geq
kyf_k^{k-2}+2ky^2f_k^{k-3}+\ldots+k(k-1)y^{k-1}.
\end{align*}
Comparing like powers of $f_k$ directly shows that this last
inequality is satisfied for $k \geq 2$ (Observe that for $k=1$ the Lemma's
claim is trivial).
\end{enumerate}
\end{proof}
\begin{remark}
Interestingly, it can also be shown that for any $2 \leq j \leq
k-1$, $T_j$ is increasing on some interval $[0,\sigma_j]$, and is
then decreasing on $[\sigma_j, 1]$.  We do not devote any space to
proving this, as it is not needed in the sequel.
\end{remark}

We next define the accumulation functions for $1\leq j\leq k$ as
\begin{equation}
\label{E:Dj} D_{k, j}(y)=D_j(y):=T_1(y)+\ldots+T_j(y).
\end{equation}
Note that by \eqref{E:flongform} we have the identity $D_k(y)=1$.  Although the individual components $T_j$ have varying behavior, their sums are much more uniform.
\begin{lemma}\label{DLemma}
The functions $D_j$ are decreasing for all $1 \leq j \leq k$.
\end{lemma}
\begin{proof}
We proceed inductively and note that $D_k=1$ is decreasing. Now
suppose that $D_j$ is decreasing. Then
\[
D'_{j-1}=T'_1+\ldots+T'_{j-1}.
\]
If $T'_i(y)\leq 0$ for all $1\leq i\leq j-1$ and all $y \in [0,1]$,
then $D_{j-1}$ is clearly decreasing. Otherwise, if there is some
$y$ and some $1\leq i\leq j-1$ such that $T'_i(y)>0,$ then we also
have that
\[
T'_j(y)>0,
\]
since $T_j=T_i\left(\frac{y}{f_k}\right)^{j-i}$, and both factors are
non-negative and  increasing at $y$. Then
\[
D'_{j-1}(y)=D'_j(y)-T'_j(y)\leq D'_j(y)\leq 0
\]
by the induction hypothesis.
\end{proof}

\begin{proof}[Proof of Proposition \ref{L:HLemma}]
We factor out the product of all of the $f_k(y_i)$ from the
definition of $H_k$ and write
\begin{align}
\label{E:Hk} H_k(y_1, \ldots, y_k) & =f_k(y_1)\ldots f_k(y_k)
\Bigg(\frac{1-y_k}{f_k(y_k)}+\frac{(1-y_{k-1})y_k}{f_k(y_{k-1})f_k(y_k)}+\ldots\\
& +\frac{(1-y_i)y_{i+1}\ldots y_k}{f_k(y_i)f_k(y_{i+1}) \ldots
f_k(y_k)}+\ldots+\frac{(1-y_1)y_2\ldots y_k}{f_k(y_1)\ldots
f_k(y_k)}-1\Bigg). \notag
\end{align}
Consider the second-to-last term, which can be written as
\begin{equation*}
\frac{y_2\ldots y_k}{f_k(y_2)\ldots f_k(y_k)}\cdot D_1(y_1).
\end{equation*}
Lemma \ref{DLemma} states that $D_1(y_1)\geq D_1(y_2)$, so we can
replace $y_1$ by $y_2$ without increasing \eqref{E:Hk}. After this
substitution, there are now two terms that contain $y_2$, and the
combination of them can be written as
\begin{equation*}
\frac{y_3\ldots y_k}{f_k(y_3)\ldots f_k(y_k)}\cdot D_2(y_2).
\end{equation*}
Again we can use Lemma \ref{DLemma} and replace $y_2$ by $y_3$
without increasing the result.  Continuing in this way, we
successively increase the indices of all of the $y_i$'s until
\eqref{E:Hk} becomes
\begin{align*}
H_k(y_1, \dots, y_k) \geq f_k(y_1) \cdots f_k(y_k) \left(D_k(y_k) -
1\right) = 0.
\end{align*}
\end{proof}

\begin{proof}[Proof of Proposition \ref{P:kgapslower}]
As mentioned earlier, we use an inductive argument based on
\eqref{E:rec}.  Without loss of generality, assume that the $u_i$
are decreasing (if they are increasing, simply reverse the order)
and recall that the proposition holds for the initial cases.

The inductive hypothesis implies that every term on the right-hand
side of \eqref{E:rec} has a lower bound in terms of the $f_k$, which
gives
\begin{align*}
\rho_{n+k}
& \geq \prod_{i=1}^n f_k(1-u_i)\Big(u_{n+k}f_k(1-u_{n+1})\dots f_k(1-u_{n+k-1})\\
& +u_{n+k-1}(1-u_{n+k})f_k(1-u_{n+1})\dots f_k(1-u_{n+k-2})+\ldots+u_{n+1}(1-u_{n+2})\dots(1-u_{n+k})\Big) \\
& = \prod_{i=1}^n f_k(1-u_i) \Big(H_k(y_{n+1}, \ldots,
y_{n+k})+f_k(1-u_{n+1})\ldots f_k(1-u_{n+k})\Big),
\end{align*}
where  $y_i:=1-u_i$, so the $y_i$ are increasing. Then by
Proposition \ref{L:HLemma} we have the required bound
\[
\rho_{n+k}  \geq \prod_{i=1}^{n+k} f_k(1-u_i).
\]
\end{proof}
\begin{remark}
Numerical evidence (and the naive heuristic of grouping together
the events with larger probabilities) suggests that Proposition
\ref{P:kgapslower} may also hold for arbitrary probabilities $u_i$,
and that the cases where the $u_i$ are monotonically decreasing or
increasing (or more generally, unimodular) give the tightest bounds.
\end{remark}

\subsection{Upper bound for $k$-gaps}
We also prove the upper bound of Theorem \ref{T:kgaps2sided} as
a separate statement for easier reference.
\begin{proposition}\label{P:kgapsupper}
Let $A_1, \ldots, A_n$ be independent events with probabilities $1
\geq u_1\geq u_2\geq \ldots \geq u_n\geq 0$. Then
\begin{equation*}
\Prob\big(\{A_i\}_{i=1}^n \text{ has no $k$-gaps }\big)\leq
\prod_{i=1}^{n-k+1}f_k(1-u_i).
\end{equation*}
\end{proposition}

As before, we will use an auxiliary function, although in order to
prove an upper bound, the function needs to have bounds that are
opposite from those proven for $H_k$. Let
\begin{align}
\label{E:tildeHk} \tilHk(y_1, y_2, \ldots, y_{2k-1})
& := (1-y_{2k-1})f_k(y_1)\ldots f_k(y_{k-1})+ (1-y_{2k-2})y_{2k-1}f_k(y_1)\ldots f_k(y_{k-2}) + \dots\\
& +(1-y_{k+i})y_{k+i+1}\cdots y_{2k-1}f_k(y_1)\dots
f_k(y_{i})+\ldots+(1-y_{k})y_{k+1}\ldots y_{2k-1} \notag \\ &
-f_k(y_1)\ldots f_k(y_{k}), \notag
\end{align}
and note that this function has ``decoupled'' the
variables that appear in polynomial terms from those that appear as
the arguments of $f_k$.  In particular, $\tilHk$ is a linear
function in each of $y_{k+1}, \dots, y_{2k-1}$.

\begin{proposition}\label{L:tildeHLemma}
For $0\leq y_1\leq \ldots\leq y_{2k-1}\leq 1$, we have that
$\tilHk(y_1, \ldots, y_{2k-1})\leq 0$.
\end{proposition}

\begin{proof}
As in the proof of Proposition \ref{P:kgapslower}, we will
eventually change all of the $y_i$'s to $y_k$ by using properties of
the functions $D_j(y)$, but this time we will also need to study the
partial derivatives of the linear terms of $\tilHk$ in order to do
so.

For the first step, consider the coefficient of $y_{k+1}$ in
\eqref{E:tildeHk}, which is
\begin{equation}
\label{E:y_{k+1}} -y_{k+2}\ldots
y_{2k-1}f_k(y_1)+(1-y_{k})y_{k+2}\ldots y_{2k-1} = y_{k+2}\ldots
y_{2k-1}f_k(y_1)\left(-1+\frac{1-y_{k}}{f_k(y_1)}\right).
\end{equation}
Since $f_k$ is decreasing, replacing $y_1$ by $y_k$ inside the
parentheses gives the following upper bound for \eqref{E:y_{k+1}}:
\begin{equation}
y_{k+2}\ldots y_{2k-1}f_k(y_1)\big(-1+D_1(y_{k})\big)\leq 0,
\end{equation}
where the inequality follows from the fact that $D_1 \leq D_2
\leq \dots \leq D_k = 1$.  This means that
\begin{equation*}
\frac{\partial}{\partial y_{k+1}} \tilHk \leq 0,
\end{equation*}
and thus changing $y_{k+1}$ into $y_{k}$ in \eqref{E:tildeHk}
increases the total expression.

Continuing inductively, assume that we have already shown that
changing $y_{k+j}\mapsto y_k$ for $1\leq j\leq i-1$ gives an upper
bound for \eqref{E:tildeHk}.  The next case is then the coefficient
of $y_{k+i}$, which is
\begin{align*}
& y_{k+i+1}\ldots y_{2k-1}\big(-f_k(y_1)\ldots
f_k(y_i)+(1-y_{k})f_k(y_1)\ldots f_k(y_{i-1})
+\ldots +(1-y_{k})y_{k}^{i-1}\big)\\
& = y_{k+i+1}\ldots y_{2k-1} f_k(y_1)\ldots f_k(y_i)
\left(-1+\frac{1-y_{k}}{f_k(y_i)}+\frac{(1-y_{k})y_{k}}{f_k(y_{i-1})f_k(y_i)}+\ldots+\frac{(1-y_{k})y_{k}^{i-1}}{f_k(y_1)\ldots f_k(y_i)}\right) \\
&  \leq y_{k+i+1}\ldots y_{2k-1}f_k(y_1)\ldots
f_k(y_i)\Big(-1+D_{i}(y_{k})\Big)\leq 0,
\end{align*}
where we used the fact that $f_k$ is decreasing in order to shift all
denominator arguments to $y_k$. Thus the partial derivative with
respect to $y_{k+i}$ is also non-positive, so the substitution
$y_{k+i} \mapsto y_k$ gives an  upper bound.  Finally, once all
$y_{k+i}$ for $1 \leq i \leq k-1$ have been set to $y_k$, we reach
the desired conclusion, as
\begin{eqnarray*}
\widetilde{H}(y_1, \ldots, y_{k}, y_{k}, \ldots, y_{k}) &\leq&
f_k(y_1)\ldots f_k(y_{k})\ \left( \frac{1-y_k}{f_k(y_k)} +
\frac{(1-y_{k})y_k}{f_k(y_{k-1})f_k(y_{k})} + \cdots +
\frac{(1-y_{k})y_k^{k-1}}{f_k(y_{1})\cdots f_k(y_{k})}  -1
\right) \\
&\leq& f_k(y_1)\ldots f_k(y_{k})\Big(D_k(y_{k})-1\Big)=0.
\end{eqnarray*}
\end{proof}

\begin{proof}[Proof of Proposition \ref{P:kgapsupper}]
As before, we use an inductive argument that relies on the
recurrence \eqref{E:rec}. By the inductive hypothesis we obtain
\begin{align*}
\rho_{n+k} & \leq
\prod_{i=1}^{n-k+1}f_k(1-u_i)\Big(u_{n+k}f_k(1-u_{n-k+2})\cdots
f_k(1-u_n) \\
& +u_{n+k-1}(1-u_{n+k})f_k(1-u_{n-k+2})
\cdots f_k(1-u_{n-1})+\ldots+u_{n+1}(1-u_{n+2})\cdots (1-u_{n+k})\Big)\\
& =\prod_{i=1}^{n-k+1}f_k(1-u_i)\Big(\tilHk(y_{n-k+2}, \ldots,
y_{n+k})+f_k(1-u_{n-k+2})\cdots f_k(1-u_{n+1})\Big) \leq
\prod_{i=1}^{n+1}f_k(1-u_i),
\end{align*}
where the conclusion is due to Proposition \ref{L:tildeHLemma}.
\end{proof}

\begin{corollary}\label{C:kgapsupper}
Assume that $A_1, \ldots, A_n$ are independent events with
probabilities $0 \leq u_1\leq u_2\leq \ldots \leq u_n\leq 1$. Then
\[
\Prob\big(\{A_i\}_{i=1}^n \text{ has no $k$-gaps}\big)\leq
\prod_{i=k}^n f_k(1-u_i).
\]
\end{corollary}

\begin{proof}
This follows directly from Proposition \ref{P:kgapsupper}
since the probabilities $u_n, u_{n-1}, \ldots, u_1$ are \\
decreasing. This gives
\begin{align*}
\Prob\Big(\{A_i\}_{i=1}^n\text{ has no $k$-gaps}\Big) & = \Prob\Big(\{A_{n-i+1}\}_{i=1}^n\text{ has no $k$-gaps}\Big) \\
& \leq \prod_{i=1}^{n-k+1}f_k\Big(1-u_{n-i+1}\Big)=\prod_{i=k}^n
f_k(1-u_i).
\end{align*}
\end{proof}

\begin{proof}[Proof of Theorem \ref{T:kgaps2sided}]
The theorem statement combines Proposition \ref{P:kgapslower} and
Corollary \ref{C:kgapsupper}.
\end{proof}

\section{Logarithmic probability estimates}
\label{S:Logprobability} In this section we largely adopt
Holroyd, Liggett, and Romik's notation from \cite{HLR} and define
 $$
 g_k(z):=-\log f_k(e^{-z}),
 $$
 which will allow us to translate products of $f_k$ (as seen in the probability estimates from Section \ref{S:Probability}) into (exponent) sums involving $g_k$.  We catalog a number of useful properties of $g_k$ that we will need throughout the rest of the paper; the $k=1$ and $k=2$ cases of Lemma \ref{L:gkproperties} were proven and used by Gravner and Holroyd in \cite{GH08, Hol03} (with some slight misstatements in parts \ref{L:gk:z0} and \ref{L:gk:g'z0}), although it should be noted that several of their proofs
 used the explicit formulas for $f_1$ and $f_2$, whereas our proofs use only the defining properties of $f_k$ for any $k$.

\begin{lemma}\label{L:gkproperties}
For all $k \geq 1$, the following are true.
\begin{enumerate}
\item
\label{L:gk:dec} The function $g_k$ is decreasing and convex.
\item
\label{L:gk:integral} The function $g_k$ is integrable on
$\R_{\geq 0}$, and satisfies the explicit evaluation
\begin{equation*}
\int\limits_{0}^\infty g_k(z) dz = \lambda_k.
\end{equation*}

\item
\label{L:gk:zinfty} As $z\to\infty$, $g_k$ has the asymptotic
behavior
\begin{equation*}
g_k(z)\sim e^{-kz}.
\end{equation*}
\item
\label{L:gk:z0} As $z\to 0$,
\begin{equation*}
g_k(z)\sim \frac{1}{k}\log z^{-1}.
\end{equation*}
\item
\label{L:gk:g'z0} As $z\to 0$,
\begin{equation*}
g'_k(z)\sim -\frac{1}{kz}.
\end{equation*}
\end{enumerate}
\end{lemma}

\begin{proof}
\begin{enumerate}
\item
This follows from Lemma 15 of  \cite{HLR}.

\item
This was  proven as Theorem 1 of \cite{HLR}, and a second shorter
proof was later given in \cite{AEPR}.

\item
Differentiate the recursion \eqref{E:fshortform} to obtain
\begin{equation}\label{E:f'}
\left(kf_k^{k-1}-(k+1)f_k^k\right)f'_k=kx^{k-1}-(k+1)x^k.
\end{equation}
The boundary value $f_k(0)=1$ implies that $f'_k(0)=0$.
Repeatedly differentiating \eqref{E:f'} then iteratively shows that
\begin{equation}\label{diffzero}
f''_k(0)=\ldots=f_k^{(k-1)}(0)=0.
\end{equation}
Finally, the $k$-th derivative yields a nonzero term, which is
\begin{equation*}
f_k^{(k)}(0)=-k!.
\end{equation*}
Therefore the Taylor expansion of $f_k$ around $0$ has the form
\begin{equation*}
f_k(x)=1-x^k+O\left(x^{k+1}\right).
\end{equation*}
This leads to the stated assertion, since
\begin{equation*}
g_k(z)=-\log f_k\left(e^{-z}\right)\sim
-\log\left(1-e^{-kz}\right)\sim e^{-kz}
\end{equation*}
as $z\to\infty$.

\item
Taking the logarithm of \eqref{E:fshortform} yields
\begin{equation}\label{E:logequation}
k\log f_k(x)+\log (1-f_k(x))=k\log x+\log(1-x).
\end{equation}
Isolating $- \log f_k\left( e^{-z}\right)= g_k(z)$, we then apply
the limiting value $f_k(1)=0$ to find that
\begin{equation*}
\lim_{z\to 0}\frac{g_k \left(e^{-z}\right)}{\log
\left(z^{-1}\right)} = \lim_{z\to
0}\frac{\frac{1}{k}\left(\log\left(1-f_k\left(e^{-z}\right)\right)-k\log
\left( e^{-z}\right)-\log \left(1-e^{-z}\right)\right)}{\log
\left(z^{-1}\right)} = \lim_{z\to 0}
\frac{\log\left(1-e^{-z}\right)}{k\log z}=\frac{1}{k}.
\end{equation*}

\item
Equation \eqref{E:logequation} implies that
\begin{equation*}
-kg_k(z)+\log
\left(1-f_k\left(e^{-z}\right)\right)=-kz+\log\left(1-e^{-z}\right).
\end{equation*}
Differentiating this equation gives
\begin{equation}
\label{E:gk'}
-kg'_k(z)+\frac{e^{-z}f'_k\left(e^{-z}\right)}{1-f_k\left(e^{-z}\right)}=-k+\frac{e^{-z}}{1-e^{-z}}.
\end{equation}
We next claim that
\begin{equation}\label{E:fder}
f_k(x)\sim (1-x)^{\frac{1}{k}}\quad\text{ as } x\to 1.
\end{equation}
This is enough to conclude \ref{L:gk:g'z0}. Indeed, by \eqref{E:f'}
we have
\[
f'_k(x)=\frac{kx^{k-1}-(k+1)x^k}{kf_k^{k-1}(x)-(k+1)f_k^k(x)}.
\]
As $z\to 0$ we then apply \eqref{E:fder} to get
\[
f'_k(e^{-z})\ll
\frac{1}{f^{k-1}\left(e^{-z}\right)}\ll\frac{1}{\left(1-e^{-z}\right)^{\frac{k-1}{k}}}\sim\frac{1}{z^{\frac{k-1}{k}}}.
\]
Therefore in \eqref{E:gk'} we have the asymptotic equality
\[
-kg'_k(z)+f'_k\left(e^{-z}\right)\sim -k+\frac{1}{z},
\]
yielding
$$
g'_k(z)\sim -\frac{1}{kz},
$$
We finish the proof by verifying \eqref{E:fder}. Taking the
limit as $z \to 0$ in \eqref{E:logequation} and using the boundary
value $f_k(1) = 0$, we find that (again, $x=e^{-z}$)
\[
k\log f_k(x)\sim\log(1-x),
\]
which gives \eqref{E:logequation}.
\end{enumerate}
\end{proof}

\section{Lower Bound}
\label{S:Lower} In this section we prove the lower bound in Theorem
\ref{T:maintheorem} for arbitrary $k$ by generalizing the
combinatorial construction used in \cite{GH08}, and then using the new
results from Section \ref{S:Logprobability} to help estimate the
corresponding probabilities.  The general idea is to consider
configurations that are {\it sufficient} for growth and that occur
with large enough probability to give the tight lower bound.  It
should be noted that this construction also gives lower bounds for
the original, non-localized $k$-percolation models, as any
configuration with localized growth starting from the origin is
clearly also sufficient for unrestricted growth (the same is not
true of the upper bound in Section \ref{S:Localupper}, as localized
growth is not a {\it necessary} condition for unrestricted growth).

We first set some notation for rectangles in $\Z^2$.  For a
rectangle $R=\{a, \ldots, c\}\times\{b, \ldots, d\}$, we denote its
{\it dimensions} by
\begin{equation*}
\dim (R):=\left(c-a+1, d-b+1\right).
\end{equation*}
We also let $R(a,b)$ denote a rectangle with dimensions $(a,b)$
whose position may or may not be specified. Moreover, we visualize
the base square at the origin as the lower-left
corner of the northeast quadrant of the lattice $\Z^2$.

Following Gravner and Holroyd's basic argument in \cite{GH08}, we
construct classes of configurations that always lead to indefinite
growth.  Figure \ref{F:DJ} illustrates the two possibilities that we
consider for the growth of a rectangle $R(a,a)$ to one of size
$R(b,b)$.  The first is ``diagonal'' growth, where $R(a,a)$ grows to
$R(a+1, a+1)$, then  to $R(a+2, a+2)$, and so on until $R(b,b)$ is
active (with deviations from the diagonal of at most distance $k$);
this sort of growth was shown to give the main (logarithmic)
term of Theorem \ref{T:maintheorem} in \cite{Hol03, HLR}.  The
second sort of growth is horizontally ``skew'' growth, where growth
proceeds first in the horizontal direction only, and then continues
in the vertical direction only.  The inclusion of the second growth
event will be enough to increase the total probability by the
claimed factor of $\exp\left(c_1 s^{-1/2}\right).$

\begin{figure}[here]
\centering \scalebox{0.7}{\input{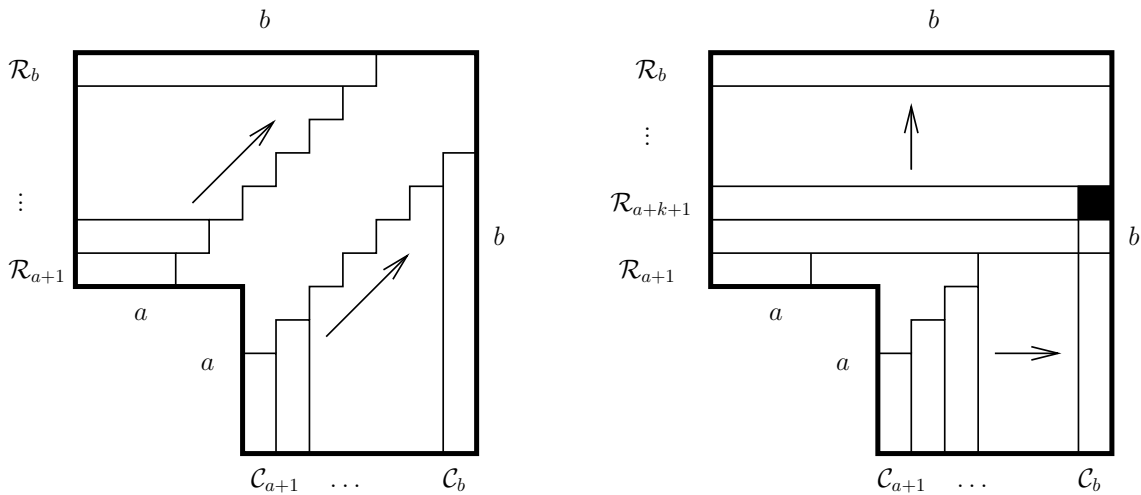}} \caption{The columns
and rows that define the events $\Dk(a,b)$ and $\Jk(a,b)$,
respectively.} \label{F:DJ}
\end{figure}

\begin{definition}
\label{D:Dk} Consider the ``stair-step'' columns and rows defined
such that for $i \geq k+1$, the column $\calC_i$ at a distance $i$
to the right of the origin has height $i - k$, and the row $\calR_i$ at
height $i$ also has width $i-k$.  If $b > a \geq k$, then the {\it
diagonal growth event} $\Dk(a, b)$ is the event that the columns
$\{\mathcal{C}_{a+1}, \ldots, \mathcal{C}_b\}$ and rows
$\{\mathcal{R}_{a+1}, \ldots, \mathcal{R}_b\}$ have no $k$-gaps in
the given configuration.
\end{definition}

\begin{definition}
\label{D:Jk} Suppose that $b - a \geq k+2$, let $\calC_{a+i}$ have
height $a+i-k$ for $1 \leq i \leq k$, and let all other $\calC_i$ have
height $a+1$ ($a+k+1 \leq i \leq b$).  For the rows, let $\calR_{a+1}$ have width $a-k+1$, let
$\calR_i$ have width $b-1$ for $a+2 \leq i \leq a+k+1$, and let
$\calR_i$ have width $b$ for $a+k+2 \leq i \leq b$.  The {\it
(horizontally) skew event}  $\Jk(a, b)$ is the event that the
following occur:
\begin{itemize}
\item $\calR_{a+1}, \calC_{a+1}, \calR_b$ and $\calC_b$ are nonempty,
\item $\calR_{a+2}, \ldots, \calR_{a+k+1}$ are empty,
\item the cell $(b, a+k+1)$ is occupied,
\item $\{\calC_{a+2}, \ldots, \calC_{b-1}\}$ and $\{\calR_{a+k+2}, \ldots, \calR_{b-1}\}$ have no $k$-gaps.
\end{itemize}
\end{definition}
It is clear that both events lead to further growth as stated in the following result.
\begin{proposition}
\label{P:DJ} Suppose that we are given a configuration $\calC$, and
consider only those rectangles whose lower-left corner is at $0$.
\begin{enumerate}
\item
If $R(a,a)$ eventually becomes active and $\Dk(a,b)$ occurs, then
$R(b-s,b-t)$ also becomes active for some $0 \leq s,t \leq k-1$.
\item
If $R(a-s, a-t)$ eventually becomes active for some $0 \leq s, t
\leq k-1$ and $\Jk(a,b)$ occurs, then $R(b,b)$ also becomes active.
\end{enumerate}
\end{proposition}

\begin{definition} \label{EkDef}
For $k\leq a_1 \leq b_1 \leq\ldots a_m\leq b_m \leq L$ with
$b_i-a_i\geq k+2$ for all $i$, define the {\it growth event}
corresponding to these parameters as
\begin{align*}
E_k(a_1, b_1, \ldots, a_m, b_m):= & \Dk(k, a_1)\cap \bigcap_{i=1}^m
\Jk(a_i, b_i)
\cap \bigcap_{i=1}^{m-1} \Dk(b_i, a_{i+1})\cap \Dk(b_m, L-1)\\
& \bigcap\left\{(k\times k) \text{ lower left rectangle and cells
}(1, L-1), (L-1, 1)\text{ are occupied}\right\}.
\end{align*}
\end{definition}

\begin{lemma}\label{L:disjointLemma}
Suppose that $\{a_i, b_i\}$ satisfy the  conditions in Definition
\ref{EkDef}.
\begin{enumerate}
\item
The various events appearing in the definition of single occurrence
of $E_k(a_1, \ldots, b_m)$ are independent.
\item
If $E_k(a_1, \ldots, b_m)$ occurs, then $R(L)$ is eventually active.
\item
For different choices of $a_1, \ldots, b_m$ the events $E_k(a_1,
\ldots, b_m)$ are disjoint.
\end{enumerate}
\end{lemma}

\begin{proof}
\begin{enumerate}
\item
Follows immediately from Definitions \ref{D:Dk} and \ref{D:Jk}.
\item
Follows from Proposition \ref{P:DJ}.
\item
Let $\mathcal{R}_n, \mathcal{R}_{n+1}, \ldots, \mathcal{R}_{n+k-1}$
be the first $k$-gap among the stair-step rows. Then
$\mathcal{R}_{n-1}$ is nonempty and begins a $\Jk$ event with $a_1 =
n-1$.  Furthermore, the event ends with the first nonempty cell in
the $(n+k-1)$-th row; the column position of this cell gives the value of $b_1$.
Following this procedure iteratively uniquely determines all $a_i$
and $b_i$.
\end{enumerate}
\end{proof}

\noindent We next bound the probabilities of the events $\Dk$ and
$\Jk$ in terms of the function $g_k$.
\begin{lemma}\label{L:DkJkestimate}
The probability of the growth events satisfies the following lower
bounds.
\begin{enumerate}
\item
If $b > a \geq k$, then
\begin{equation*}
\displaystyle \Prob\left(\Dk(a,b)\right) \geq \exp\left(-2 \sum_{i=a
- (k-1)}^{b-k}g_k(is)\right).
\end{equation*}
\item
Let $c_-<c_+$ be positive constants, $s\in \left(0,
\frac{1}{2}\right)$, and $b \geq a + k + 2$, with $a, b\in
\left[c_-s^{-1}, c_+s^{-1}\right]$. Then
\begin{equation*}
\Prob\left(\Jk(a, b)\right) \mathop{\gg}_{\text{\emph{unif}}} s \: \exp\left(g'_k(c_-)s(b-a)^2 - 2
\sum_{i=a - (k-1)}^{b-k}g_k(is) \right),
\end{equation*}
where the asymptotic inequality is uniform over all $a, b$ in the given range.
\end{enumerate}
\end{lemma}
\begin{proof}
\begin{enumerate}
\item
This follows directly from Proposition \ref{P:kgapslower} and the
definitions of $\Dk(a,b)$ and $g_k$.
\item
From the definition of $\Jk(a,b)$ and Proposition \ref{P:kgapslower}
we obtain
\begin{align*}
\Prob\big( & \Jk(a, b)\big) \geq q^{k(b-1)}\left(1-q\right)\left(1-q^{a-k+1}\right)^2
\left(1-q^{a+1}\right)\left(1-q^b \right) \\
& \quad \times
\exp\Big(-g_k((a-k+2)s) - \dots - g_k(as) -\left(b-a-k-1\right)g_k\left((a+1)s\right) \\ 
& \qquad \qquad \qquad \qquad \qquad \qquad - \left(b-a-k-2\right)g_k(bs)\Big),
\end{align*}
where the first $q$-power is for empty rows, the next several factors are for
occupied rows and columns, and the final exponential terms are for the gap
conditions among the remaining rows and columns. In the given ranges
of $a, b$, the powers $q^a$ and $q^b$ may be treated as (uniform) asymptotic
constants, as can the single terms $g_k((a-k+i)s)$ (of which there are $k-2$).  
Also, $1-q$ is asymptotically $s$, so the overall 
bound becomes
\begin{align}
\label{E:PJk} \Prob &\Big(\Jk(a, b)\Big) \\
& \mathop{\gg}_{\text{unif}}  s \: \exp\Big(-(b-a)\Big(g_k\left((a+1)s\right)+g_k(bs)\Big) + (k+1) g_k\left((a+1)s\right)+(k+2)g_k(bs)\Big) \notag \\
& \geq s \: \exp\Big(-(b-a)\left(g_k(as)+g_k(bs)\right)\Big). \notag
\end{align}
The second inequality holds since the function $g_k$ is decreasing. It is also true that
\[
\exp\left(-\sum\limits_{i=a-(k-1)}^{b-k}g_k(is)\right)\leq
\exp(-(b-a)g_k(bs)),
\]
again since $g_k$ is decreasing.  Using this in \eqref{E:PJk}, we have
\begin{align*}
\Prob\big(\Jk(a, b)\big)
& \mathop{\gg}_{\text{unif}}  s\exp\left(-(b-a)\left(g_k(as)-g_k(bs)\right) -2 \sum\limits_{i=a-(k-1)}^{b-k}g_k(is)\right) \\
& \geq s\exp\left((b-a)g_k'(c_-)(bs-as) -2
\sum\limits_{i=a-(k-1)}^{b-k}g_k(is)\right),
\end{align*}
where we have used the convexity of $g_k$ for the final
approximation.
\end{enumerate}
\end{proof}

With the combinatorial preliminaries finished, we now prove the
lower bound by selecting a ``window'' of size proportional to $s^{-1}$.
\begin{proof}[Proof of lower bound in Theorem \ref{T:maintheorem}]
Let $m:=\left\lfloor s^{-\frac{1}{2}}M\right\rfloor$, where $M < 1$
is a positive constant that will be chosen later, and suppose that
we have a sequence of parameters that satisfy
\begin{equation*}
s^{-1}<a_1\leq b_1\leq \ldots\leq a_m\leq b_m < \left\lfloor 2s^{-1}
\right\rfloor =: L,
\end{equation*}
with $b_i-a_i \in \left[k+2, s^{-\frac{1}{2}}\right]$ for all $i$.
Lemma \ref{L:DkJkestimate} shows that there is a constant $c>0$
such that  we have the following lower bound for the probability of a
growth event (note that the asymptotic bounds are uniform across all $a_i, b_i$):
\begin{align}
\label{E:ProbE}
\Prob & \big(E_k(a_1,\ldots, b_m)\big) \\
& =\Prob\big(\Dk(k, a_1)\big)\prod_{i=1}^m \Prob\big(\Jk(a_i,
b_i)\big)\prod_{i=1}^{m-1}\Prob\big(\Dk(b_i, a_{i+1})\big)
\Prob\big(\Dk(b_m, L)\big)\Prob\left(k^2+2\text{ active sites}\right) \notag \\
& \mathop{\gg}_{\text{unif}} s^{k^2+2}\exp\Bigg(-2\sum\limits_{i=1}^{L-k}g_k(is)\Bigg)s^m c^m \exp\Bigg(-cs\sum\limits_{i=1}^m(b_i-a_i)^2\Bigg) \notag \\
& \mathop{\gg}_{\text{unif}}  s^m c^m \exp\left(-2\sum\limits_{i=1}^{L-k}g_k(is) - (k^2 + 2)
\log s^{-1}\right), \notag
\end{align}
where for the last estimate we absorbed the last factor into the constant power $c^m$ by using the fact that
$b_i-a_i \in \left[k+2, s^{-\frac{1}{2}}\right]$. The number of
possible sequences $\{a_i, b_i\}$ is at least
\begin{align}
\label{E:choose-fct} \left(
\begin{matrix}
\left\lfloor s^{-1}-m s^{-\frac{1}{2}}\right \rfloor\\
m
\end{matrix}
\right) & \left(s^{-\frac{1}{2}}-(k+2)\right)^m \gg \left(
\begin{matrix}
\left\lfloor s^{-1}(1-M)\right\rfloor\\
m
\end{matrix}
\right)
\left(s^{-\frac{1}{2}}\right)^m \\
& \gg \left(s^{-1}(1-M)\right)^m
\left(\frac{s^{-\frac{1}{2}}}{m}\right)^m \gg
s^{-m}\left(\frac{1-M}{M}\right)^m, \notag
\end{align}
where the second approximation comes from Stirling's formula.

Finally, note that by Lemma \ref{L:disjointLemma} the event $E_k(a_1,
\ldots, b_m)$ only guarantees growth out to $R(L, L)$.  In order to
achieve indefinite growth, we add the event $\Dk(L, \infty)$ as
well, which means (in a slight abuse of notation) that there are no
$k$-gaps in $\{\calC_{L+1}, \calC_{L+2}, \dots\}$ and
$\{\calR_{L+1}, \calR_{L+2}, \dots\}$.  This means that the entire
northeast quadrant will become active, and we achieve indefinite
growth in the whole plane by restricting further to the probability $1$ condition
that there are no empty semi-infinite lines in $\Z^2$ (this same
argument was used by Gravner and Holroyd).

Combining \eqref{E:ProbE} and \eqref{E:choose-fct}, we obtain the estimate
\begin{align*}
\Prob(\text{indefinite growth}) & \geq
\sum_{\substack{\text{sequences } \{a_i,
b_i\}}}\Prob\big(E_k(a_1,\ldots, b_m)\big) \cdot
\Prob\big(\Dk(L,\infty)\big) \\
& \gg s^{-m}\left(\frac{1-M}{M}\right)^m s^m c^m \exp\Big(-2\lambda_k s^{-1} - (k^2 + 2) \log s^{-1}\Big)\\
& \gg \left(c \cdot \frac{1-M}{M}\right)^m \exp\Big(-2\lambda_k
s^{-1} - (k^2 + 2) \log s^{-1}\Big).
\end{align*}
Choosing $M$ sufficiently
small so that $c \cdot \frac{1-M}{M}>1$, we obtain the lower bound
\begin{equation*}
\exp\left(-2\lambda_k s^{-1}+cs^{-\frac{1}{2}}\right),
\end{equation*}
which completes the proof.
\end{proof}

\section{Local upper bound}
\label{S:Localupper}

We now turn to the upper bound in Theorem \ref{T:maintheorem}. Part
of our proof follows Gravner and Holroyd's approach to the cases
$k=1$ and $k=2$, and we improve several of their choices of
parameters in order to achieve a tighter second-order term.  We
proceed through the technical preliminaries with unspecified
parameters in order to show that our final choices are optimal for
this approach. Our more significant contribution is a new
combinatorial characterization of necessary growth conditions that
allows us to adapt Gravner and Holroyd's scaling arguments to the
case of general $k$.

We introduce ``rectangle growth sequences'' in order to encode the row and column
conditions that occur in growing configurations.  These sequences naturally contain
a generalization of Gravner and Holroyd's ``good sequences'' from \cite{GH08}.  Using some intricate
combinatorial arguments, the probability of such a subsequence can
be bounded, as can the total possible number of subsequences,
and the combination of these estimates leads to the overall upper bounds.

For $k=1$, some of the arguments in Sections \ref{S:Uppercomb} --
\ref{S:Upperproof} only apply to the modified model, and we explain
the minor changes that are necessary for the $k=1$ Frob\"ose model
in Section \ref{S:UpperFrob}.

\subsection{Preliminary combinatorial setup}
\label{S:Uppercomb} We begin with the unspecified parameters; there
are several important rough asymptotic properties that we will need
for these parameters, so we define and list them now.
\begin{definition}
\label{D:param} The parameters $A$ ({\it lower} dimension), $B$
({\it upper} dimension), and $D$ ({\it growth ratio}) are assumed to
be positive values that satisfy the following limiting relations as
$s \rightarrow 0$:
\begin{center}
\begin{tabular}{|l||c|c|c|} \hline
Parameter name & $A$ & $B$ & $D$ \\ \hline Limiting value &
$\infty$ & $\infty$ & $0$ \\ \hline Asymptotics &  $\log(As) \sim
\log s$ & $\log \left(A^{-1} B\right) \ll \log s^{-1}$ &   $BD \gg 0$ \\ \hline
Inequalities & $A < B/2$  & $B \geq s^{-1} \log s^{-1}$ & $D < 1$ \\
\hline
\end{tabular}
\end{center}
\end{definition}

Next, we define the new combinatorial structure that we will use to
encode and approximate the spread of active sites in the
$k$-percolation model. \noindent
\begin{definition}
A \emph{rectangle growth sequence} for an initial configuration
$\calC$ on $\Z^2$ is denoted by $\calS(\calC)$, and is defined to be
a sequence of rectangles
\[
\calS(\calC):= \; \left\{0=R'_1\subsetneq \dots \subsetneq R'_m
\subsetneq \dots\right\}
\]
such that
\begin{enumerate}
\item
$S(\calC)$ is empty if the origin is not active in $\calC$, and
otherwise $R'_1 = \{0\}.$
\item
Each $R'_i$ has no $k$-gaps; in other words, there are no $k$
consecutive empty rows or columns in $R'_i$.
\item
$R'_{i+1}\setminus R'_i$ is contained in $\Shell(R'_i)$, where
$\Shell(R)$ is defined to be the width $1$ boundary around any
rectangle $R$.
\end{enumerate}
\end{definition}
\noindent Note that there may be many possible choices for the
sequence $\calS(\calC)$ depending on $\calC$, and furthermore, that
the sequence may be either finite or infinite.  However, we are
primarily interested in the class of {\it maximal} growth sequences
in the current study, as we need to encode the fact that the percolation process proceeds so long as
any growth is possible.  Here maximality is defined in terms of the
partial ordering of growth sequences given by (rectangle)
containment, and it is straightforward to see that there is a
well-defined ``join'' operation: if $R'_1, R'_2, \dots$ and $S'_1,
S'_2, \dots$ are both growth sequences, then both are contained in
the growth sequence $\left< R'_1, S'_1\right>, \left<R'_2,
S'_2\right>, \dots,$ where for two rectangles $R$ and $S$,
$\left<R,
S\right>$ denotes the {\it span} of $R$ and $S$ (the smallest
rectangle that contains both $R$ and $S$).

\begin{definition}
A {\it good configuration} is an initial state on $\Z^2$ such that
any maximal rectangle growth sequence is infinite.
\end{definition}
\noindent This concept is useful for proving upper bounds for
growth, as the next lemma shows that good configurations are a
characterizing property of indefinite growth.
\begin{lemma}
\label{L:growthsequence} If $\calC$ has indefinite growth, then
$\calC$ is a good configuration.
\end{lemma}
\begin{proof}
We argue by contradiction and suppose that $\calC$ is not a good
configuration, and thus has a finite maximal growth sequence
$\calS(\calC)= \; \left\{0=R'_1\subsetneq \dots \subsetneq R'_n
\right\}$.
  By definition, such a sequence ends
with a rectangle $R'_n$ that is empty in its $k-1$ outermost rows and columns.  Furthermore,
$\Shell(R'_n)$ must also also be empty in $\calC$.  Thus
$U := R'_n \cup \Shell(R'_n)$ is empty in its $k$ outermost rows and columns, so that
$R'_{n-k+1}$ is the last rectangle in $\calS$ that has occupied
sites on its boundary.

According to the growth rules, there cannot be any active squares outside of $\calS$ until there is first an
active square somewhere in the $k$ outermost rows and columns of $U$.  However, since these rows and columns are completely empty, they will remain so even if $R'_{n-k+1}$ becomes completely active.
This contradicts the assumption that $\calC$ has indefinite growth,
thus completing the proof.
\end{proof}

Because of Lemma \ref{L:growthsequence}, we can use good
configurations as an upper bound for indefinite growth, and we spend
much of the remainder of this section showing that the arguments in
\cite{GH09} can still be applied to these configurations.  We
further classify good configurations into two types of behavior
that, much like the events $\Jk$ and $\Dk$ in Section \ref{S:Lower},
(roughly) correspond to whether the growth is ``skew'' or
``diagonal''.
\begin{definition}
A growth sequence \emph{escapes} if there is an $R'_i$ with
dimensions $(a', b')$ such that $a'\in [B, B+1]$ and $b'\leq A$, or
such that $a'\leq A$ and $b'\in [B, B+1]$.
\end{definition}

We can now generalize Gravner and Holroyd's concept of a ``good
sequence'' by considering appropriate subsequences of a growth
sequence.
\begin{definition}
\label{D:goodseq} A \emph{good sequence} is a sequence of rectangles
$0\in R_1\subsetneq\ldots\subsetneq R_{n+1}$ that satisfies the
following conditions on the dimensions $\dim R_i=(a_i, b_i)$:
\begin{enumerate}
\item
$\min\{a_1, b_1\}\in [A, A+1]$
\item
$a_n+b_n\leq B$
\item
$a_{n+1}+b_{n+1}>B$
\item
For $i=1, \ldots, n$ we have $s_i\geq a_i D$ or $t_i\geq b_i D$,
where $s_i:=a_{i+1}-a_i$ and $t_i:=b_{i+1}-b_i$ are the successive
dimension differences.
\item
For $i=1, \ldots, n$ we have $s_i<a_i D+2$ and $t_i<b_i D+2$.
\end{enumerate}
\end{definition}

\noindent For a rectangle $R$ define next the event
\begin{equation*}
G(R):=\{R \text{ has no $k$-gaps in columns or rows}\}.
\end{equation*}
Furthermore for two rectangles $R\subseteq R'$, we define the
subrectangles $S_1, \ldots, S_8$ (some of which may be empty) as in
Figure \ref{F:RS}.
\begin{figure}[here]
\centering \scalebox{0.6}{\input{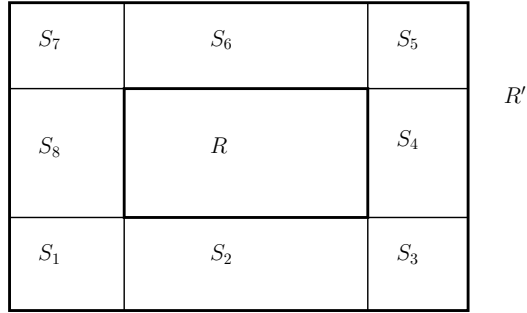}} \caption{The
rectangular regions defined by $R$ and $R'$} \label{F:RS}
\end{figure}

\begin{definition}
Let $D(R, R')$ denote the event that each of the two rectangles
$S_1\cup S_8\cup S_7$ and $S_3\cup S_4\cup S_5$ have no $k$-gaps
along the columns, and that each of the two rectangles $S_1\cup
S_2\cup S_3$ and $S_7\cup S_6\cup S_5$ has no $k$-gaps along the
rows.
\end{definition}
\begin{remark}
One easily sees that $D(R,R')$  is necessary for the growth to
proceed from $R$ to $R'$.
\end{remark}

\begin{lemma}
\label{L:escape} A good configuration $\calC$ either has a good
sequence $R_1   \subsetneq \ldots  \subsetneq R_{n+1}$ such that
$G(R_1)$ and $\displaystyle\bigcap_{i=1}^n D(R_i, R_{i+1})$ occur,
or $S(\calC)$ escapes.
\end{lemma}

\begin{proof}
If $S(\calC)$ does not escape, then $R_1  \subsetneq \ldots
\subsetneq R_{n+1}$ can be taken as a subsequence of $S(\calC)$ that is
determined solely by the rectangle dimensions.  All of the
conditions in Definition \ref{D:goodseq} are easily seen to be
satisfied.
\end{proof}

\subsection{Probability estimates}
\label{S:Upperestimates} We now approximate the probability of
various events involving good sequences and follow the technical
framework used in \cite{GH08, GH09}, again using our new combinatorial
definitions to unify the results for all $k$.  Most of the proofs
are straightforward generalizations of Gravner and Holroyd's, using
our general estimates from Section \ref{S:Logprobability}, but we include
all technical steps in order to be precise with certain ``$k$-shifts'' that occur.

We begin with a lemma that was proven in \cite{HLR}, which is also
the special case where all probabilities are equal in the upper bound of our
Proposition \ref{P:kgapsupper}.
\begin{lemma} \label{L:kgap}
If $R = R(a,b)$ is a rectangle, then
\begin{equation*}
 \Prob(R \text{ has no } k \text{-gaps in its columns or rows})
\leq
\begin{cases}
\exp \big(-(a-(k-1)) g_k(bs)\big) & \qquad \text{ if }a\leq b,\\
\exp \big(-(b-(k-1)) g_k(as)\big) & \qquad \text{ if }a\geq b.
\end{cases}
\end{equation*}
\end{lemma}

Next we prove $k$-analogs of several results in \cite{GH09},
beginning with the probability that the growth sequence
$\calS(\calC)$ escapes.
\begin{lemma}\label{L:escapelemma}
If $s$ is sufficiently small, then there exists
 a constant $c>0$
such that
\[
\Prob\Big(S(\mathcal{C})\text{ escapes}\Big)\leq \exp\left(-cB\log
s^{-1}\right).
\]
\end{lemma}
\begin{proof}
The number of possible rectangles $R$ such that the growth sequence
escapes is at most $4A^2(B+1)$. For such an $R$ Lemma \ref{L:kgap}
implies that as $s\to 0$,
\[
\Prob\left(G(R)\right)\leq \exp\Big(-(B-(k-1))g_k(As)\Big)\leq
\exp\Big(-c\cdot B g_k(As)\Big)
\]
for some constant $c$; we obtain this uniform bound by using the
fact that $g_k$ is decreasing (Lemma \ref{L:gkproperties}
\ref{L:gk:dec}). By the assumptions in Definition \ref{D:param},
$As\to 0$ for $s\to 0$, so we may use Lemma \ref{L:gkproperties}
\ref{L:gk:z0}, giving that,
\[
\Prob\Big(S(\mathcal{C})\text{ escapes}\Big)\leq 4 A^2 (B+1) \exp
\left(-\frac{Bc}{k}\log(As)^{-1}\right)\leq \exp \left(-c B\log
s^{-1}\right).
\]
The final bound follows since we can absorb the leading factor into
the exponential error term, and also from the asymptotic
\[
\log(As)^{-1}\sim\log s^{-1}.
\]
\end{proof}

Next come several bounds related to good sequences, beginning with a
uniform bound for the initial rectangle $R_1$.
\begin{lemma}\label{L:entryLemma}
Let $R_1, \ldots, R_{n+1}$ be a good sequence of rectangles and let
$a_0=b_0=A$, $s_0=a_1-a_0$, $t_0=b_1-b_0$. Then we have for some
constant $c>0$
\[
\Prob\left(G(R_1)\right)\leq
\exp\Big(-s_0g_k(b_0s)-t_0g_k(a_0s)+cA^{-1}B\Big).
\]
\end{lemma}
\begin{proof}
We assume without loss of generality that $a_1\geq b_1$, so
$b_1\in[A, A+1]$. By Lemma \ref{L:kgap} we have
\[
\Prob\left(G(R_1)\right)\leq
\exp\Big(-\left(a_1-(k-1)\right)g_k(b_1s)\Big).
\]
Thus, noting that $t_0\leq 1$,
\begin{equation}\label{PGR}
\begin{split}
\frac{\Prob\left(G(R_1)\right)}{\exp\Big(-s_0g_k(b_0s)-t_0g_k(a_0s)\Big)}
& \leq \exp\Big(-\left(a_1-(k-1)\right)g_k(b_1s)+s_0 g_k(As)+ t_0 g_k(As)\Big)\\
& \leq \exp\Big(s_0\left(g_k(As)-g_k(b_1s)\right)-(A-(k-1))g_k (b_1s)+g_k(As)\Big)\\
& \leq \exp\Big(s_0\left(g_k(As)-g_k(b_1s)\right)+g_k(As)\Big).
\end{split}
\end{equation}
 Since $g_k$ is convex and monotonically decreasing (Lemma
\ref{L:gkproperties}  \ref{L:gk:dec}),
\[
0\leq g_k(As)-g_k(b_1 s)\ll -s g'_k(As)\ll A^{-1},
\]
where for the last estimate we used Lemma \ref{L:gkproperties} (v).
Moreover, Lemma \ref{L:gkproperties} \ref{L:gk:z0} gives
\[
g_k(As)\ll\frac{1}{k}\log s^{-1}.
\]
Thus \eqref{PGR} can be estimated against
\[
\exp(cs_0A^{-1}+c\log s^{-1})\leq \exp \left( cBA^{-1}\right),
\]
where the final inequality follows from our assumption that $BA^{-1}
\gg \log{s^{-1}}.$
\end{proof}

\begin{lemma}\label{L:DR}
If $R\subseteq R'$ are two rectangles with dimensions $(a, b)$ and
$(a+\ell, b+m),$ respectively, then
\[
\Prob\left(D(R, R')\right)\leq \exp
\Big(-\left(m-2(k-1)\right)g_k(as)-\left(\ell-2(k-1)\right)g_k(bs)
+\ell ms \exp\big(k\left(g_k(as)+g_k(bs)\right)\big) \Big).
\]
\end{lemma}
\begin{proof}
We use the same notation as in Figure 2 and   split the event
according to the total number of occupied sites in the corners
$S_1\cup S_3\cup S_5 \cup S_7$. The probability that exactly $j$
(out of a possible number of $\ell m$) such corner sites are
nonempty is
\begin{equation*}
\binom{\ell m}{j} (1-q)^j q^{\ell m-j}.
\end{equation*}
The presence of these occupied corner sites divides $S_4$ and $S_8$ into
column strips $\calC_1, \dots, \calC_\alpha$ for some $\alpha \leq j
+ 2$, where the strip $\calC_i$ has corresponding width $c_i$. Using
Lemma \ref{L:kgap} we obtain a bound for the event that $S_4$ and
$S_8$ have no $k$-gaps; namely
\begin{align*}
\Prob & \left(S_4 \text{ and } S_8 \text{ have no } k \text{
gaps}\right)
 \leq  \exp\Big(-\left(c_1-(k-1)\right)g_k(bs)-\ldots-\left(c_\alpha-(k-1)\right)g_k(bs)\Big) \\
& \qquad \qquad = \exp\Big(-\left(c_1+\ldots+c_\alpha-\alpha(k-1)\right)g_k(bs)\Big)\\
&  \qquad \qquad \leq
\exp\Big(-\left(\ell-j-(j+2)(k-1)\right)g_k(bs)\Big)=
\exp\Big(-\left(\ell-kj-2(k-1)\right)g_k(bs)\Big).
\end{align*}
A similar argument applies to the rows in $S_2$ and $S_6$, and the
column and row events are independent. Thus
\begin{equation*}
\Prob\left(D(R, R')\right)\leq \sum_{j=0}^{\ell m} \binom{\ell m}{j}
(1-q)^j q^{\ell m-j}
\exp\Big(-\left(\ell-kj-2(k-1)\right)g_k(bs)-\left(m-kj-2(k-1)\right)g_k(as)\Big).
\end{equation*}
The approximations $q\leq 1$ and $1-q\leq s$ then imply that
\begin{align*}
\Prob & \left(D(R, R')\right)
\leq \exp\Big(-\left(\ell-2(k-1)\right)g_k(bs)-\left(m-2(k-1)\right)g_k(as)\Big) \\
& \qquad \qquad \qquad \qquad \qquad \times \sum_{j=0}^{\ell m}
\binom{\ell m}{j}
s^j \exp\Big(kj\left(g_k(bs)+g_k(as)\right)\Big)\\
& =
\exp\Big(-\left(\ell-2(k-1)\right)g_k(bs)-\left(m-2(k-1)\right)g_k(as)\Big)
\Big(1+s \exp\left(k\left(g_k(bs)+g_k(as)\right)\right)\Big)^{\ell m}\\
& \leq
\exp\Big(-\left(\ell-2(k-1)\right)g_k(bs)-\left(m-2(k-1)\right)g_k(as)+s\ell
m \exp\left(k\left(g_k(as)+g_k(bs)\right)\right)\Big),
\end{align*}
where for the last inequality we used the crude estimate $1+x\leq
e^x$.
\end{proof}
Several of the prior papers in this subject have used a general variational result for probabilities involving
convex functions \cite{GH09, Hol03}, and we bound the resulting error terms for $g_k$
and our particular parameters.
\begin{lemma}[Lemma 7 of \cite{GH09}]\label{L:convexLemma}
Suppose that $A$ and $B$ are positive integers satisfying $2A < B$,
and that $(a_i, b_i)_{i=1, \ldots n+1}$ satisfy $a_0=b_0=A$ and
$s_i:=a_{i+1}-a_i\geq 0$ and $t_i:=b_{i+1}-b_i\geq 0$, as well as
the first three properties of Definition \ref{D:goodseq}.  For any
$s>0$ and any positive, smooth, convex, decreasing function $g:(0,
\infty)\to (0, \infty)$,
\begin{equation*}
\sum_{i=1}^n\Big(s_i g(b_i s)+t_i g(a_i s)\Big) \geq
\frac{2}{s}\int_{As}^{Bs} g(t)dt-2Bg\left(\frac{Bs}{2}\right).
\end{equation*}
\end{lemma}
\begin{remark}
The parameters in this statement are slightly shifted from those used by Gravner and Holroyd,
but the proof is analogous.
\end{remark}
\begin{corollary}
\label{C:gsumbound} If $A$ and $B$ are as in Definition
\ref{D:param}, then for $s$ sufficiently small there exists a
constant $c>0$ such that
\begin{equation*}
\sum_{i=1}^n\Big(s_i g_k(b_i s)+t_i g_k(a_i s)\Big) \geq 2\lambda_k
s^{-1} - c\left( A \log{s^{-1}} - s^{-1} \exp(-kBs) -
B\exp(-kBs/2)\right).
\end{equation*}
\end{corollary}
\begin{proof}
We first consider the integral in the bound of Lemma
\ref{L:convexLemma}, and write
\begin{equation*}
\int_{As}^{Bs}g_k(z)dz=\int_0^\infty g_k(z)dz-\int_0^{As}g_k(z)dz-
\int_{Bs}^\infty g_k(z)dz.
\end{equation*}
The first integral equals $\lambda_k$ by Lemma \ref{L:gkproperties}
\ref{L:gk:integral}.  Using Lemma \ref{L:gkproperties} \ref{L:gk:z0}
and our assumptions on $A$, the second integral can be estimated by
\begin{equation*}
c As\log (As)^{-1}\sim c As\log s^{-1}.
\end{equation*}
Lemma \ref{L:gkproperties} \ref{L:gk:zinfty} and assumptions on $B$ imply that
the third integral can be estimated by
\[
c\int_{Bs}^\infty e^{-kz}dz\sim c e^{-kBs}.
\]
Finally, for the second term of Lemma \ref{L:convexLemma}, Lemma
\ref{L:gkproperties} \ref{L:gk:zinfty} again gives that as $s\to 0$,
\[
B g_k\left(\frac{Bs}{2}\right)\ll B \exp\left(-\frac{k}{2}
Bs\right).
\]
\end{proof}

Next we consider bounds involving the dimensions of a good sequence.
\begin{lemma}\label{L:nestimate}
Let $n$ and $a_i, b_i\: (i=1, \ldots, n+1)$ be positive integers and
denote the successive differences by $s_i:= a_{i+1}-a_i\geq 0$ and
$t_i:=b_{i+1}-b_i\geq 0$ for $i=1, \ldots, n$. Further assume that
the dimensions satisfy all of the properties of a good sequence.
Then for $s\to 0$ the following bounds are satisfied:
\begin{enumerate}
\item
$\displaystyle n \ll D^{-1}\log s^{-1},$
\item
$\displaystyle \sum_{i=1}^n\frac{s_i t_i}{a_i b_i} \ll D\log
s^{-1}.$
\end{enumerate}
\end{lemma}
\begin{proof}
\begin{enumerate}
\item
To bound $n$, we use (\emph{i}), (\emph{ii}), and (\emph{iv}) from
Definition \ref{D:goodseq}
\[
(1+D)^{n-1}\leq
\frac{a_n}{a_{n-1}}\frac{a_{n-1}}{a_{n-2}}\ldots\frac{a_2}{a_1}\frac{b_n}{b_{n-1}}\frac{b_{n-1}}{b_{n-2}}\ldots\frac{b_2}{b_1}=\frac{a_n}{a_1}\frac{b_n}{b_1}\ll
\frac{B^2}{A^2}.
\]
Taking logarithms yields
\[
(n-1)\log(1+D)\ll \log\left(\frac{B}{A}\right)\ll \log s^{-1}.
\]
Thus
\[
n D\ll \log s^{-1},
\]
which gives the claim.
\item
Simple uniform bounds and part (\emph{i}) show that
\[
\sum_{i=1}^n\frac{s_i t_i}{a_i b_i}\ll D^2n\ll D\log s^{-1}.
\]
\end{enumerate}
\end{proof}
The above results are enough to bound the probability of any good
sequence, and we further estimate the total number of such
sequences.
\begin{lemma}\label{L:entropy}
The number of good sequences of rectangles is at most
\[
\exp\left(cD^{-1}\left(\log s^{-1}\right)^2\right),
\]
where $c>0$ is some constant.
\end{lemma}

\begin{proof}
First, there are at most $4B\cdot (A+1)B$ choices for $R_1$.  Next,
given $R_i$ there are less than $(BD)^4$ choices for $R_{i+1}$. The length of the sequence is $n+1$,
where $n$ is bounded by Lemma \ref{L:nestimate}, and we
have an overall bound of
\[
\ll (A+1)B^2(BD)^{4n+4}\leq \exp\left(c n\log B\right)\leq
\exp\left(cD^{-1}\left(\log s^{-1}\right)^2\right).
\]
\end{proof}

We end with one additional technical estimate.
\begin{lemma}
\label{L:gkaz} Suppose that $B$ satisfies the preceding assumptions.
Then for the range $0 \leq a \leq B$, there is a uniform asymptotic bound
\begin{equation*}
e^{g_k(as)}\mathop{\ll}_{\text{unif}}  \left(\frac{B}{a}\right)^{\frac{1}{k}} \qquad
\text{as } s \to 0
\end{equation*}
\end{lemma}
\begin{proof}
\item
As $z\to 0$, Lemma \ref{L:gkproperties} \ref{L:gk:z0} implies that
\[
e^{g_k(z)}\ll \exp \Big(\frac{1}{k}\log
z^{-1}\Big)
=\frac{1}{z^{\frac{1}{k}}},
\]
and as $z\to\infty$, Lemma \ref{L:gkproperties} \ref{L:gk:zinfty}
gives
\[
e^{g_k(z)}\ll \exp\left(e^{-kz}\right)\ll 1.
\]
Recall that $Bs \rightarrow \infty$ as $s \rightarrow 0.$  By continuity and the above asymptotics, there is thus a sufficiently large $M$ and constant $c$ such that if $z \leq M$, then
\begin{equation*}
e^{g_k(z)}\leq c \left(\frac{M}{z}\right)^{\frac{1}{k}}.
\end{equation*}
\end{proof}

\subsection{Proof of the upper bound in Theorem \ref{T:maintheorem}}
\label{S:Upperproof} We are now ready to prove the main result of
this section.  By Lemma \ref{L:escape}, we have
\begin{equation}
\label{E:Pindef} \Prob\left(\text{indefinite growth}\right)\leq
\Prob\left(\calS(\calC) \text{ escapes}\right)+
\sum_{\substack{\text{good sequences} \\ {R_1, \ldots, R_{n+1}}}}
\Prob\left(G(R_1)\right)\prod_{i=1}^n \Prob\left(D\left(R_i,
R_{i+1}\right)\right).
\end{equation}
We note for future reference that we will use Lemma
\ref{L:escapelemma} to estimate $\Prob(S(C)\text{ escapes})$.  For
the good sequences term, Lemma \ref{L:DR} says that
\begin{align*}
\Prob& \left(D\left(R_i, R_{i+1}\right)\right) \\
& \mathop{\ll}_{\text{unif}}
\exp\Big(-\left(t_i-2(k-1)\right)g_k(a_is)-\left(s_i-2(k-1)\right)g_k(b_is)\Big)
\exp\Big(s_i t_i s\exp \left(k
\left(g_k(a_is)+g_k(b_is)\right)\right)\Big) .
\end{align*}
Applying the uniform bound from Lemma \ref{L:gkaz} gives
\begin{equation}
\label{E:PD} \Prob\left(D\left(R_i, R_{i+1}\right)\right)\mathop{\ll}_{\text{unif}}
\exp\Big(-\left(t_ig_k(a_is)+s_i
g_k(b_is)\right)\Big)\left(\frac{B^2}{a_i
b_i}\right)^{\frac{2(k-1)}{k}} \exp\left(s_i t_i
s\left(\frac{B^2}{a_i b_i}\right)\right).
\end{equation}
We also use Lemma \ref{L:entryLemma} to bound $\Prob(G(R_1))$ and
Lemma \ref{L:entropy} to bound the number of good sequences.
Combined with \eqref{E:PD}, this gives the following upper bound for
the second term in \eqref{E:Pindef}:
\begin{align}
\label{E:Pindef2nd}
& \exp\left(cD^{-1}\left(\log s^{-1}\right)^2-s_0 g_k(b_0s)-t_0 g_k(a_0s)+c A^{-1}B\right)\\
& \qquad \qquad \qquad \qquad \times \prod_{i=1}^n
\exp\Big(-\left(t_i g_k(a_is)+ s_i
g_k(b_is)\right)\Big)\left(\frac{B^2}{a_i
b_i}\right)^{\frac{2(k-1)}{k}}
\exp\Bigg(s_i t_i s\left(\frac{B^2}{a_i b_i}\right)\Bigg) \notag \\
& \mathop{\ll}_{\text{unif}} \exp\left(-\sum_{i=0}^n\big(s_i g_k(b_is)+ t_i
g_k(a_is)\big)\right)\exp\Big(cD^{-1}(\log s^{-1})^2+cA^{-1}B\Big)
\left(\frac{B^2}{A^2}\right)^{\frac{2(k-1)}{k}n} \notag \\
& \qquad \qquad \qquad \qquad \qquad \qquad \times
\exp\left(B^2s\sum_{i=1}^n\frac{s_i t_i}{a_i b_i}\right).  \notag
\end{align}
Using Corollary \ref{C:gsumbound}, the first (multiplicative) term
in \eqref{E:Pindef2nd} is bounded above by
\begin{equation*}
 \exp\left(-\frac{2\lambda_k}{s}+c A \log s^{-1}+c s^{-1}\exp \left(-k Bs\right)\right) \exp\Bigg(Bc \exp\left(-\frac{k}{2}Bs\right)\Bigg).
\end{equation*}
Next, Lemma \ref{L:nestimate} bounds the third term as
\[
\left(\frac{B^2}{A^2}\right)^{cD^{-1}\log s^{-1}}\leq \exp
\left(cD^{-1}\left(\log s^{-1}\right)^2 \right).
\]
Finally, the last term of \eqref{E:Pindef2nd} is also bounded by
Lemma \ref{L:nestimate}, giving
\begin{equation*}
\exp\left(B^2s\sum_{i=1}^n\frac{s_i t_i}{a_i b_i}\right) \leq \exp
\left(B^2 D s \log s^{-1}\right).
\end{equation*}

Combining all of these approximations gives the following upper
bound for the second summand of \eqref{E:Pindef2nd}:
\begin{align}
\label{E:Pindef2bound}
\exp&\Big(-2\lambda_k s^{-1}+c\left(A\log s^{-1}+A^{-1}B+D^{-1}\left(\log s^{-1}\right)^2+s B^2D\log s^{-1}\right) \\
& \qquad \qquad \qquad \qquad +c\left(s^{-1}\exp\left(-kBs\right)+B
\exp\left(-\frac{kBs}{2}\right)\right)\Big). \notag
\end{align}
By assumption $B \gg s^{-1}\log s^{-1}$, and thus the last two terms
are simply part of the error.  Examining the rest of the expression
shows that this error is in fact optimized when $B\sim s^{-1}\log
s^{-1}$.  We now assume that we can write
\[
A\sim s^{-\alpha}\left(\log s^{-1}\right)^\gamma\qquad D\sim
s^{\beta}\left(\log s^{-1}\right)^\delta
\]
for some positive constants $\alpha, \beta$ and some constants
$\gamma, \delta$, and tabulate the corresponding powers that arise
from the relevant terms in \eqref{E:Pindef2bound}:
\begin{center}
\begin{tabular}{|l||c|c|c|c|} \hline
Term & $A\log s^{-1}$ & $A^{-1}B$ & $D^{-1}\left(\log
s^{-1}\right)^2$ & $s B^2D\log s^{-1}$ \\ \hline $s$-power &
$-\alpha$ & $-1+\alpha$ & $-\beta$ & $-1+\beta$ \\ \hline $\log
s^{-1}$-power & $1+\gamma$ & $1-\gamma$  & $2-\delta$ & $3+\delta$
\\ \hline
\end{tabular}
\end{center}
 We first consider the
$s$-powers, and easily see that the optimal choices are
$\alpha=\beta=\frac{1}{2}$.
 Turning to the $\log s^{-1}$-powers, we
find that the best choices are $\gamma=0$, $\delta=-\frac{1}{2}$,
and that the second summand of \eqref{E:Pindef} can be bounded by
\[
\exp\left(-2\lambda_ks^{-1}+cs^{-\frac{1}{2}}\left(\log
s^{-1}\right)^{\frac{5}{2}}\right).
\]
Furthermore, recalling Lemma \ref{L:escapelemma} and our discussion
of the parameter $B$, we have the competing bound
\[
\Prob\left(\calS(\calC) \text{ escapes}\right)\ll
\exp\left(-cs^{-1}\left(\log s^{-1}\right)^2\right).
\]
We can now prove the theorem statement, as
\begin{align*}
\Prob\left(\text{indefinite growth}\right) &\leq
\exp\Big(-2\lambda_k
s^{-1}\Big)\Bigg(\exp\Big(s^{-1}\left(2\lambda_k-c\left(\log
s^{-1}\right)^2\right)\Big)
+\exp\Big(cs^{-\frac{1}{2}}\left(\log s^{-1}\right)^{\frac{5}{2}}\Big)\Bigg)\\
& \ll \exp\Big(-2\lambda_k s^{-1}+cs^{-\frac{1}{2}}\left(\log
s^{-1}\right)^{\frac{5}{2}}\Big).
\end{align*}

\subsection{Upper bound for the Frob\"ose Model}
\label{S:UpperFrob} We end our study of percolation models with the
Frob\"ose model and prove the remaining $k=1$ case of Theorem
\ref{T:maintheorem}. It is only necessary to briefly mention the
difference in comparison with the $k=1$ modified case. As before a
growing configuration has a rectangle growth sequence, but in this
case each shell not only has to be nonempty, but must also have
non-corner occupied sites. Thus growth happens in only one direction
at a time, which allows us to use disjointedness and the van der
Berg -- Kesten (BK) inequality rather than the corner decomposition
of Lemma \ref{L:DR}. To be more precise (using the same notation as
before), if $R$ grows to $R'$, then the disjoint intersection of
events
\[
\left(\bigcap_{i=1}^m \calR_i\text{ nonempty }\right) \circ
\left(\bigcap_{i=1}^\ell\calC_i \text{ nonempty }\right)
\]
occurs (here disjointedness means that it is possible to choose $m +
\ell$ distinct nonempty cells, one for each column and row). The BK
inequality \cite{Gri} then implies that
\[
\Prob\left(D(R_i, R_{i+1})\right) \leq
\left(1-q^{b_i+t_i}\right)^{s_i}\left(1-q^{a_i+s_i}\right)^t.
\]
Therefore
\begin{align*}
& \sum\limits_{\substack{\text{good sequences}\\ {R_1, \ldots,
R_{n+1}}}}
\Prob\left(G(R_1)\right)\prod_{i=1}^n \Prob\left(D\left(R_i, R_{i+1}\right)\right)\\
& \ll \exp\Big(D^{-1}\left(\log
s^{-1}\right)^2-2s^{-1}\left(\lambda_1-As\log
s^{-1}-\exp(-Bs)\right)
+B\exp(-sB/2)+A^{-1}B\Big)\\
& \ll \exp\Big(-2\lambda_1 s^{-1}+c\left(D^{-1}\left(\log
s^{-1}\right)^2+A \log s^{-1}+A^{-1}B\right) +c\left(s
\exp(-Bs)+B \exp(-Bs/2)\right)\Big).
\end{align*}
As before we choose $B=s^{-1}\log s^{-1}$ and write
$A=s^{-\alpha}\left(\log s^{-1}\right)^\gamma$, $D=s^\beta\left(\log
s^{-1}\right)^\delta$. We see that the error is optimized for
$\alpha=\frac{1}{2}$, $\beta=\frac{1}{2}$, $\gamma=0$, and
$\delta=1$. This gives us a savings of $\log s^{-1}$ over Gravner
and Holroyd's arguments.

\section{Proof of Theorem \ref{T:probTheorem}}
\label{S:PAk}

We next turn to the improved probability bound. Using Theorem
\ref{T:kgaps2sided} with $u_i=1-e^{-is}$ yields
\begin{equation*}
\prod_{j=1}^n f_k\left(e^{-js}\right)\leq
\Prob\left(\{A_j\}_{j=1}^n\text{ has no $k$-gaps}\right) \leq
\prod_{j=k}^n f_k\left(e^{-js}\right).
\end{equation*}
Since the events form a decreasing, nested sequence, we may take the limit as
$n\to\infty$. This gives
\begin{equation}
\label{E:Pinftyfk} \exp\left(-\sum\limits_{j=1}^\infty
g_k(js)\right)\leq \Prob\left(\{A_i\}_{j=1}^\infty\text{ has no
$k$-gaps}\right)\leq \exp\left(-\sum\limits_{j=k}^\infty
g_k(js)\right).
\end{equation}

We now use the Integral Comparison Theorem, which states that if
$h(z)$ is a decreasing, convex function such that $\displaystyle
\lim_{z \to \infty} h(z) = 0$, then we have
\begin{equation*}
\frac{h(1)}{2} + \int_1^\infty h(z) dz\leq \sum_{j=1}^\infty h(j)
\leq \int_0^\infty h(z) dz.
\end{equation*}
This gives
\begin{equation}
\label{E:Pinftygk} \exp\left(-\int\limits_0^\infty
g_k(zs)dz\right)\leq
 \Prob \left( \{A_j\}_{j=1}^\infty\text{ has no $k$-gaps}\right)
 \leq \exp\left(-\frac{g_k(ks)}{2} -\int\limits_0^\infty g_k(zs)dz+\int\limits_0^k g_k(zs)dz\right).
\end{equation}
Recall the integral evaluation from Lemma \ref{L:gkproperties} (ii),
and make the substitution $w=zs$.  Then the lower bound in
\eqref{E:Pinftygk} is simply
\begin{equation*}
 \Prob\left(\{A_j\}_{j=1}^\infty\text{ has no $k$-gaps}\right) \geq \exp\left(-\int\limits_0^\infty g_k(w)\frac{dw}{s}\right)
 = \exp\left(-\lambda_k s^{-1}\right).
\end{equation*}
The upper bound becomes
\begin{equation*}
\Prob\left(\{A_j\}_{j=1}^\infty\text{ has no $k$-gaps}\right) \leq
\exp\left(-\lambda_k s^{-1} -\frac{g_k(ks)}{2} +
\frac{1}{s}\int\limits_0^{ks}g_k(w)dw\right).
\end{equation*}
By Lemma \ref{L:gkproperties}  (iv), this has the asymptotic
behavior
\begin{align*}
\Prob\left(\{A_j\}_{j=1}^\infty\text{ has no $k$-gaps}\right) & \leq
\exp\left(-\lambda_k s^{-1} - \frac{1}{2k} \left(1 + o(1)\right)
\log s^{-1} + s\log s^{-1}\left(1+o(1)\right) \right),
\end{align*}
where the last term follows from the integral estimate
\begin{align*}
& \int\limits_{0}^{ks}g_k(w)dw \leq
\int\limits_{0}^{ks}\left(1+o(1)\right)\left(-\frac{1}{k}\log
w\right)dw
 =\left(-\frac{1}{k}+o(1)\right)\left(w\log w-w \Big|_{0}^{ks}\,\right)\\
& =\left(-\frac{1}{k}+o(1)\right) \big(ks\log(ks)-ks\big)=-s\log
s\left(1+o(1)\right).
\end{align*}
This gives
\begin{equation}
\label{E:PAkexp} \exp\left(-\lambda_k s^{-1}\right)\leq
\Prob\left(\{A_j\}_{j=1}^\infty\text{ has no $k$-gaps}\right)\leq
s^{-\frac{(2k-1)}{2k}}\left(1+o(1)\right) \cdot \exp\left(-\lambda_k
s^{-1}\right),
\end{equation}
as claimed.

\begin{remark}
The upper and lower bounds in \eqref{E:Pinftyfk} are easily seen to
differ by a factor of at most $s^{-(k-1)/k}$ in the asymptotic
limit.  By using the Integral Comparison Theorem to write the final
bound in the form of \eqref{E:PAkexp}, we have introduced the
additional error factor of $s^{-1/2k}$.
\end{remark}

\end{document}